\documentclass[12pt]{amsart}
\usepackage{amsmath,amssymb,amsthm}
\usepackage{color}    
\usepackage{hyperref}
\usepackage{graphicx}   
\usepackage{wrapfig}

\usepackage{amsfonts}
\usepackage{graphics}
\usepackage{pstricks}
\usepackage{verbatim}
\usepackage{tikz}

\usepackage{soul}
\usepackage{ulem}

\newtheorem{thm}{Theorem}[section]

\newtheorem{lem}[thm]{Lemma}
\newtheorem{prop}[thm]{Proposition}
\newtheorem{cor}[thm]{Corollary}

\theoremstyle{definition}
\newtheorem{definition}{Definition}[section]

\theoremstyle{remark}

\theoremstyle{remark}
\newtheorem{remark}{Remark}[section]
\numberwithin{equation}{section}

\newcommand{\N}{{\mathbb N}}

\newcommand{\R}{{\mathbb R}}

\definecolor{blu}{rgb}{0,0,1}

\title[Multiplicity result]{Multiplicity results in the non-coercive case for an elliptic problem
with critical growth in the gradient}

\author[C. De Coster]{Colette De Coster }%$^*$}
\address{Colette de Coster
\newline\indent
Universit\'e de Valenciennes et du Hainaut Cambr\'esis
\newline\indent
LAMAV,  FR CNRS 2956, 
\newline\indent
Institut des Sciences et Techniques de Valenciennes
\newline\indent
F-59313  Valenciennes Cedex 9, France}
\email{Colette.DeCoster@univ-valenciennes.fr}

\author[L. Jeanjean]{Louis Jeanjean}%$^*$}
\thanks{L. Jeanjean is supported by the project NONLOCAL (ANR-14-CE25-0013), funded by the French National Research Agency (ANR)}
\address{Louis Jeanjean
\newline\indent
Laboratoire de Math\'ematiques (UMR 6623)
\newline\indent
Universit\'{e} de Franche-Comt\'{e}
\newline\indent
16, Route de Gray 25030 Besan\c{c}on Cedex, France} 
\email{louis.jeanjean@univ-fcomte.fr}

\begin{document}
\subjclass[2010]{35J25, 35J62}

\keywords{Quasilinear elliptic equations, quadratic growth in the gradient, lower and upper solutions}

\begin{abstract}
We consider the boundary value problem
\begin{equation*}
- \Delta u = \lambda c(x)u+ \mu(x) |\nabla u|^2 + h(x), \qquad u \in H^1_0(\Omega) \cap L^{\infty}(\Omega),
\leqno{(P_{\lambda})}
\end{equation*}
where $\Omega \subset \R^N, N \geq 3$ is a bounded domain with smooth boundary. 
It is assumed that $c\gneqq 0$, $c,h$ belong to $L^p(\Omega)$ for some  $p > N$. Also  $\mu \in L^{\infty}(\Omega)$ and $\mu \geq \mu_1 >0$ for some $\mu_1 \in \R$. It is known that when $\lambda \leq 0$, problem $(P_{\lambda})$ has at most one solution. In this paper we study, under various assumptions, the structure of the set of solutions of  $(P_{\lambda})$  assuming that $\lambda>0$. Our study unveils the rich structure of this problem. We show, in particular, that what happen for $\lambda=0$ influences the set of solutions in all the half-space $]0,+\infty[\times(H^1_0(\Omega) \cap L^{\infty}(\Omega))$. Most of our results are valid without assuming that $h$ has a sign. If we require $h$ to have a sign, we observe that the set of solutions differs completely for $h\gneqq 0$ and $h\lneqq 0$. We also show  when $h$ has a sign  that solutions not having this sign may exists. Some uniqueness results of signed solutions are also derived. The paper ends with a list of open problems.
\end{abstract}

\maketitle

\section{Introduction}

We consider the boundary value problem
\begin{equation*}
- \Delta u = \lambda c(x)u+ \mu(x) |\nabla u|^2 + h(x), \qquad u \in H^1_0(\Omega) \cap L^{\infty}(\Omega) 
\leqno{(P_{\lambda})}
\end{equation*}
under the assumption
$$
%\hspace{1cm}
\left\{ \begin{array}{c} 
\Omega \subset \R^N,\,\,  N \geq 2 \,\,\mbox{ is a bounded domain with } \partial \Omega
\mbox{  of class } C^{1,1},
\\[2mm]
c \mbox{ and } h \mbox{ belong to }  L^p(\Omega) \,\, \mbox{for some } p > N \mbox{ and satisfy } c\gneqq 0,
\\[2mm]
 \mu \in L^{\infty}(\Omega) \mbox{ satisfies } 0<\mu_1\leq \mu(x)\leq \mu_2.
\end{array}
\right.
\leqno{\mathbf{(A)}}
$$
\medbreak
Depending on the parameter $\lambda \in \R$ we study the existence and multiplicity of solutions of  $(P_{\lambda})$. By solutions we mean functions $u\in H_0^1(\Omega)\cap L^{\infty}(\Omega)$ satisfying 
$$
\int_\Omega \nabla u \nabla v \, dx
				= 
\lambda \int_\Omega c(x)  u  v \, dx
		+
\int_\Omega \mu(x)  |\nabla u |^2 v \, dx
		+
\int_\Omega h(x) v \, dx \, , 
$$
for any $v\in H_0^1(\Omega)\cap L^{\infty}(\Omega).$

First observe that, by the change of variable $v=-u$, problem $(P_{\lambda})$ reduces to
\begin{equation*}
- \Delta u = \lambda c(x)u - \mu(x) |\nabla u|^2 - h(x), \qquad u \in H^1_0(\Omega) \cap L^{\infty}(\Omega).  
\end{equation*}
Hence, since we make no assumptions on the sign of $h$, we actually also consider the case where $ |\nabla u|^2$ has a negative coefficient.
\medbreak

The study of quasilinear elliptic equations with a gradient dependence up to the critical growth $|\nabla u|^2$ was essentially initiated by
Boccardo, Murat and Puel in the 80's and it has been an active field of research until now. Under the condition
$\lambda c(x) \leq -\alpha_0 <0 $ a.e. in $\Omega$ for some $\alpha_0 >0$, which is usually referred to as the {\it coercive case},
the existence of a unique solution of $(P_{\lambda})$  is guaranteed by assumption (A). 
This is a special case of the results of \cite{BoMuPu1,BoMuPu3} for the existence and of \cite{BaBlGeKo,BaMu}  for the uniqueness.  

The limit case where one just require that $\lambda c(x) \leq 0$  a.e. in $\Omega$  is more complex. There had been a lot of contributions  \cite{AbDaPe, %AlLiTr,
FeMu1, MaPaSa,Po} when   $\lambda = 0$ (or equivalently when $c\equiv 0$) but the general case where 
$\lambda c \leq 0 $  may vanish only on some parts of $\Omega$ was left open until the paper \cite{ArDeJeTa}. 
It appears in \cite{ArDeJeTa} that under assumption (A) the existence of solutions is not guaranteed, 
additional conditions are necessary. When $\lambda =0$ this was already observed in \cite{FeMu1}. 
By \cite{ArDeJeTa}, the uniqueness itself holds as soon as $\lambda c(x) \leq 0$ a.e. in $\Omega$. See also \cite{ArDeJeTa2}  
for a related uniqueness result in a more general frame.

The case $\lambda c \gneqq 0$ remained unexplored until very recently. Following the paper \cite{Si} which consider a particular case, Jeanjean and Sirakov \cite{JeSi} study a problem directly connected to $(P_{\lambda})$. In \cite[Theorem 2]{JeSi} assuming that $\mu$ is a positive constant and $h$ is small (in an appropriate sense) but without sign condition, a $\lambda_0 >0$ is given under which $(P_{\lambda})$ has two solutions whenever $\lambda \in \,]0, \lambda_0[$. This result have been complemented in \cite{HuJe} where two solutions are obtained, 
allowing the function $c$ to change sign but assuming that $h \geq 0$ and that  $\max\{0, \lambda c\} \gneqq  0$. The restriction that $\mu$ is a constant was subsequently removed in \cite{ArDeJeTa} under the price  of the assumption $h \geq 0$.

If multiplicity results can be observed in case  $\lambda c \gneqq 0$, the existence of solution itself may  fail. In \cite[Lemma 6.1]{ArDeJeTa}, letting $\gamma_1 >0$ be the first eigenvalue of 
\begin{equation}
\label{eigenvaluep}
%\begin{array}{ccc}
-\Delta \varphi_{1} = \gamma c (x) \varphi_{1},  \quad \varphi_1 \in H^1_0(\Omega),
%\\
% \varphi_{1} =0,  &\mbox{ on }& \partial\Omega
%\end{array}
\end{equation}
it is proved when $h \geq 0$ that problem $(P_{\lambda})$ has no solution when $\lambda = \gamma_1$ and no non-negative solutions when $\lambda > \gamma_1$. 
This contrasts to what was observed  in \cite[Theorem 3.3]{AbPePr}, namely that if $\mu >0$ is a constant and  $h\lneqq 0$, then there exists a negative solution of 
$(P_{\lambda})$ as soon as $\lambda>0$. In addition this negative solution is unique \cite[Theorem 3.12]{AbPePr}. Considered together, the results of \cite{AbPePr,ArDeJeTa} 
show that the sign of $h$ has definitely an influence on the set of solutions of 
$(P_{\lambda})$ when $\lambda >0$. 
\medskip

Despite the works \cite{AbPePr,ArDeJeTa,HuJe, JeSi}, having a clear picture of the set of solutions of $(P_{\lambda})$ in the half-space 
$]0,+\infty[\times(H^1_0(\Omega) \cap L^{\infty}(\Omega))$ is still widely open. The present paper aims to be a contribution in that direction. 
Note that both in \cite{AbPePr} and \cite{ArDeJeTa} the main results (under this assumption) are obtained assuming that $h$ has a sign, positive in \cite{ArDeJeTa}, 
negative in \cite{AbPePr} and then these papers  look for solutions having the same sign as $h$. In our paper we remove in particular the assumption that $h$ has a sign.
Also we show that even when $h$ has a sign, solutions not having this sign may exist. \medskip

We point out that with respect to \cite{AbPePr, ArDeJeTa} we have strengthened our regularity assumptions by requiring  $c$ and $h$ in $L^p(\Omega)$ 
for some $p>N$ while in  \cite{ArDeJeTa},  $c$ and $h$ are in $L^p(\Omega)$ for some   $p> \frac{N}{2}$ and in \cite{AbPePr}, the regularity assumptions are even weaker. 
Under our assumptions all solutions of $(P_{\lambda})$ lies in $W^{2,p}_0(\Omega) \subset C^1_0(\overline{\Omega})$ (see Theorem \ref{regularity}).  
This permits to use lower and upper solutions arguments together with degree theory.  Now for future reference we recall,

\begin{definition}
Let $u,v \in C(\overline{\Omega})$.  We say that 
\\
$\bullet$ $u\leq v$ if, for all $x\in \Omega$, $u(x)\leq v(x)$;
\\
$\bullet$ $u\lneqq v$ if, for all $x\in \Omega$, $u(x)\leq v(x)$ and $u\not\equiv v$;
\\
$\bullet$ $u < v$ if, for all $x\in \Omega$,  $u(x) < v(x)$.
\end{definition}

Let $\varphi_1$ be the first eigenfunction of \eqref{eigenvaluep}. We know that, for all $x\in \Omega$,  $\varphi_1(x) > 0$ and, for $x \in \partial \Omega$,
$\frac{\partial \varphi_1}{\partial \nu} (x)< 0$ where $\nu$ denotes the exterior unit normal.

\begin{definition}
Let $u,v \in C(\overline{\Omega})$.  We say that 
\\
$\bullet$ $u \ll v$ in case there exists $\varepsilon>0$ such that, for all $x\in \overline\Omega$, $v(x)-u(x)\geq \varepsilon \varphi_1(x)$.
\end{definition}

\begin{remark}
Observe that, in case $u,v \in C^1(\overline{\Omega})$, the definition of  $u \ll v$ is equivalent to: for all $x\in \Omega$,  $u(x) < v(x)$
 and, for $x \in \partial \Omega$,
either $u(x) < v(x)$ or $u(x) = v(x)$ and $\frac{\partial u}{\partial \nu} (x)
> \frac{\partial v}{\partial \nu} (x)$.
\end{remark}

Recall that by \cite[Theorems 1.2 and 1.3]{ArDeJeTa}, we have the following result relying on \cite[Theorem 3.2]{Ra}.

\begin{thm} 
\label{ADJT1}
Under assumption $(A)$, for $\lambda\leq 0$ the problem  $(P_{\lambda})$ has at most one solution $u_{\lambda}$. Moreover, in case $(P_{0})$ has a solution $u_0$, then 
$$
\Sigma = \{ (\lambda, u) \in \R \times C(\overline \Omega) 
\mid (\lambda, u) \mbox{ solves } (P_{\lambda})\},
$$
possesses one unbounded component ${\mathcal C}^+$ in $[0, + \infty[ \times C(\overline \Omega)$ such that ${\mathcal C}^+\cap (\{0\}\times C(\overline\Omega))=\{u_0\}$.

In case $h\gneqq  0$,  this continuum 
${\mathcal C}^+$ consists of non-negative functions and
its projection  $\mbox{\rm Proj}_{\R} {\mathcal C}^+$ 
on the $\lambda$-axis is an interval  $]-\infty,\overline \lambda] \subset {]-\infty,\gamma_1[}$ containing $\lambda =0$  and  
${\mathcal C}^+ \subset \Sigma$ bifurcates from infinity to the right of the axis $\lambda = 0$. 
\end{thm}
%
%Moreover, there exists $\lambda_0\in {]0,\overline \lambda]}$
%such that
%for all $\lambda \in {]0, \lambda_0[}$, the section ${\mathcal C}\cap (\{\lambda\}\times C(\overline\Omega))$
% contains two distinct non negative solutions of $(P_{\lambda})$ in $\Sigma$. \medskip

\begin{remark}
From \cite[Corollary 3.2]{ArDeJeTa}, we know that $(P_0)$  has a solution if
\begin{equation}
\label{cond ArDeJeTa}
\displaystyle \inf_{ \{ u \in H^1_0(\Omega) \mid \,\,\|u\|_{H^1_0(\Omega)}=1 \}} \,  \displaystyle \int_{\Omega} \left(|\nabla u|^2 -  \mu_2 h^+(x) u^2 \right) dx >0,
\end{equation}
where $h^+ = \max\{0,h\}$.
\end{remark}
 
Our first main result gives informations on the behaviour of this continuum without assuming  that $h \gneqq 0$.

\begin{thm}
\label{thm 0}
Under assumption $(A)$, in case $(P_{0})$ has a solution, the continuum  ${\mathcal C}^+$ of Theorem \ref{ADJT1} satisfies  one of the two cases :
\begin{enumerate}
\item[(i)] it bifurcates  from infinity to the right of the axis $\lambda = 0$ with the corresponding solutions having a positive part blowing up to infinity
 as $\lambda \to 0^+$; 
\item[(ii)] it is such that 
its projection  $\mbox{\rm Proj}_{\R} {\mathcal C}^+$ 
on the $\lambda$-axis is $[0,+\infty[$.
\end{enumerate}
\end{thm}

In Corollary \ref{Cor 4.1} below, we show that we are in situation (i) of Theorem \ref{thm 0} if $(P_0)$ has a solution and
$$
\int_{\Omega} h \,\varphi_1 \geq 0.
$$

In \cite[Theorem 2]{JeSi} under conditions insuring that $(P_0)$ has a solution it was proved, assuming that $\mu$ is a constant, that $(P_{\lambda})$ has two solutions 
for $\lambda >0$ small. Here we remove this restriction on $\mu$.

\begin{thm}
\label{thmlocal}
 Under assumption $(A)$ and assuming that $(P_0)$ has a solution $u_0$, there exists a $\overline\lambda \in \,]0,+\infty]$ such that 
\begin{enumerate}
\item[(i)] for every $\lambda \in \,]0,\overline\lambda[$, the problem $(P_{\lambda})$ has at least two solutions with 
\\
$\bullet$
$u_{\lambda,1} \ll u_{\lambda,2}$;
\\
$\bullet$
$\displaystyle\max_{\overline\Omega} u_{\lambda,2}\to + \infty$ and $u_{\lambda,1}\to u_0$ in $C^1_0(\overline\Omega)$ as $\lambda \to 0$;
\item[(ii)] if $\overline\lambda < + \infty$,    the problem $(P_{\overline\lambda})$ has exactly one solution $u$. 
\end{enumerate} 
\end{thm}

Next we show that having a sign information on the solution $u_0$  of  $(P_0)$ allows us to give more precise informations on the set of solutions of $(P_{\lambda})$ when $\lambda >0$.

\begin{thm}
\label{thm 1}
Under assumption $(A)$  and assuming that $(P_0)$ has a solution $u_0\geq0$ with $cu_0 \gneqq 0$, every non-negative solution of $(P_{\lambda})$ with $\lambda >0$ satisfies 
$u \gg u_0$. Moreover, there exists $\overline\lambda\in \,]0,+\infty[$ such that
\begin{enumerate}
\item[(i)] for every $\lambda \in \,]0,\overline\lambda[$, the problem $(P_{\lambda})$ has at least two solutions with 
\\
$\bullet$
$0\leq u_0 \ll u_{\lambda,1} \ll u_{\lambda,2}$;
\\
$\bullet$
if $\lambda_1< \lambda_2$, we have $u_{\lambda_1,1}\ll u_{\lambda_2,1}$;
\\
$\bullet$
$\displaystyle\max_{\overline\Omega} u_{\lambda,2}\to + \infty$ and $u_{\lambda,1}\to u_0$ in $C^1_0(\overline\Omega)$ as $\lambda \to 0$;
\item[(ii)] the problem $(P_{\overline\lambda})$ has exactly one non-negative solution $u$;
\item[(iii)] for every $\lambda>\overline\lambda$, the problem $(P_{\lambda})$ has no {\rm non-negative} solution. 
\end{enumerate} 
\end{thm}

\begin{figure}[t]
\begin{center}
\includegraphics[scale=0.3]{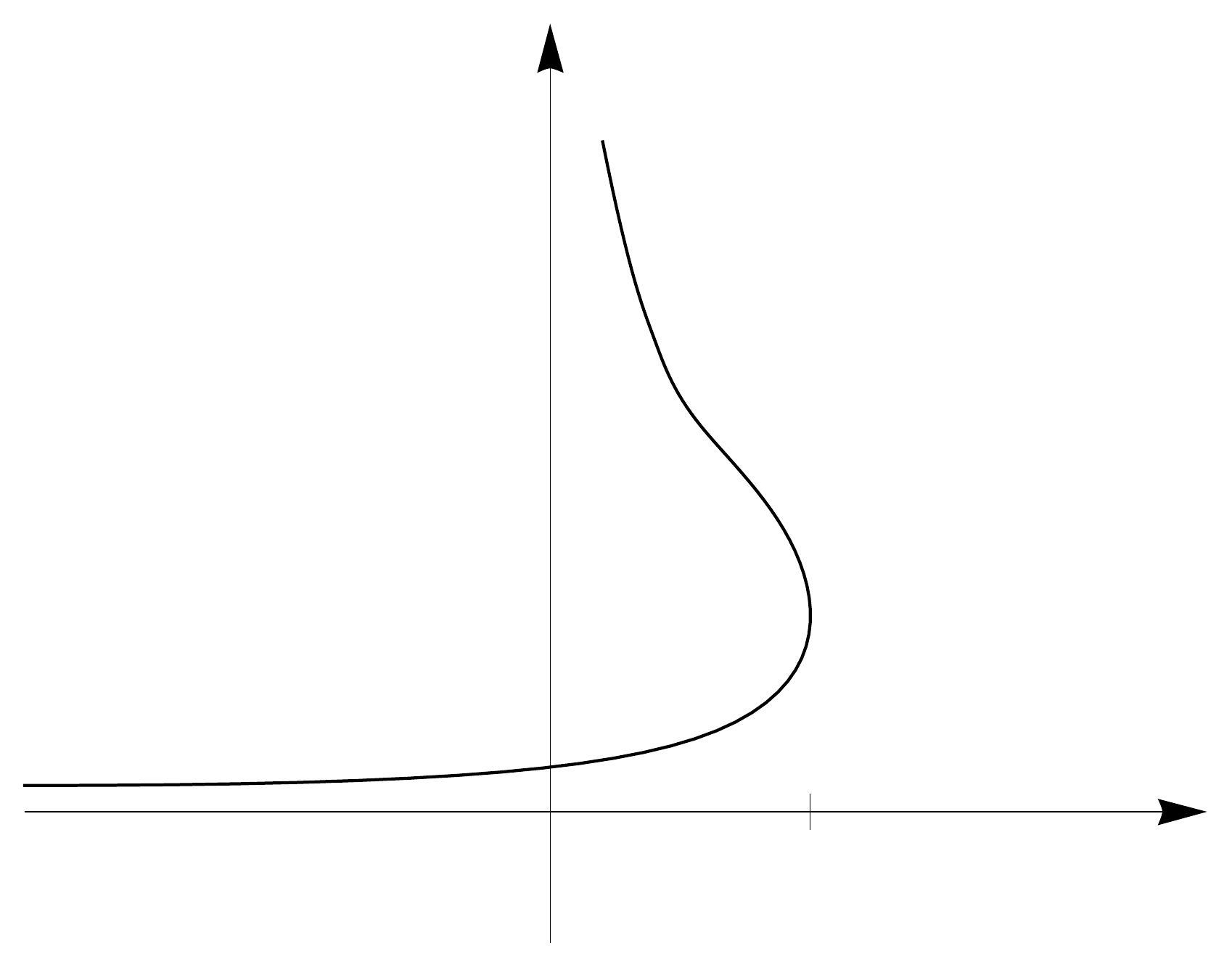}
\put(-95,27){$u_0$}
\put(-5,22){$\lambda$}
\put(-50,05){$\overline\lambda$}
\caption{Illustration of Theorem \ref{thm 1}}
\label{figdn1}
\end{center}
\end{figure}

%\textcolor{red}{inserer figure 1}

\begin{remark}
%Recording that problem  $(P_0)$ has a solution $u_0$ if \eqref{cond ArDeJeTa} holds, 
Since $- \Delta u_0 = \mu(x) |\nabla u_0|^2 + h(x)$, we deduce by the strong maximum principle that, in case $h\gneqq 0$, we have $u_0 \gg0$ thus $c u_0\gneqq 0$.
\end{remark}

In comparison to Theorem \ref{thm 1} we have 

\begin{thm}
\label{thm 2}
Under assumption $(A)$  and assuming that $(P_0)$ has a solution $u_0\leq0$ with $c u_0 \lneqq  0$, for every $\lambda >0$, problem $(P_{\lambda})$ has  two solutions with 
$$
u_{\lambda,1} \ll u_{\lambda,2}, \qquad u_{\lambda,1} \ll u_0 ,\quad \mbox{and} \quad \max_{\overline\Omega} u_{\lambda,2}>0.  
$$ 
Moreover we have
\\
$\bullet$  if $\lambda_1< \lambda_2$, then $u_{\lambda_1,1}\gg u_{\lambda_2,1}$;
\\
$\bullet$ $\displaystyle\max_{\overline\Omega} u_{\lambda,2}\to + \infty$ and $u_{\lambda,1}\to u_0$ in $C^1_0(\overline\Omega)$ as $\lambda \to 0$;
\end{thm}

\begin{figure}[t]
\begin{center}
\includegraphics[scale=0.3]{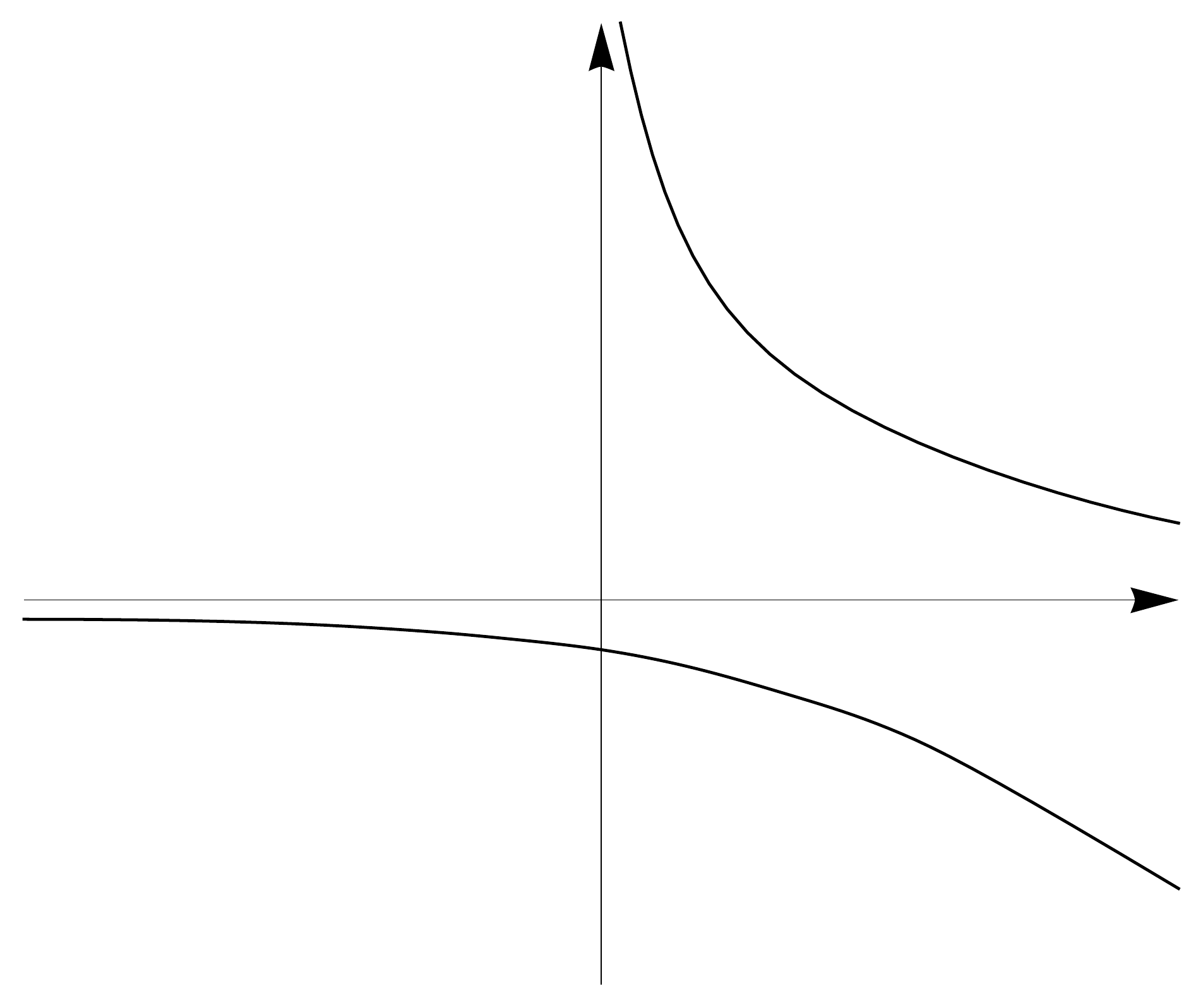}
\put(5,50){$\lambda$}
\put(-95,40){$u_0$}
\caption{Illustration of Theorem \ref{thm 2}}
\label{figdn2}
\end{center}
\end{figure}

\begin{remark}
\label{rem 0}
Observe that in case $(P_0)$ has a solution $u_0$ with  $cu_0 \equiv 0$, then $u_0$ is solution for all $\lambda\in \mathbb R$.
\end{remark}

\begin{remark}
In Proposition \ref{unique}, we prove also that, if $(P_0)$ has a solution $u_0\leq0$ with $c u_0 \lneqq  0$, then $(P_{\lambda})$ has at most one solution $u\leq 0$.
\end{remark}

\begin{cor}
\label{Cor Negatif}
Under assumption $(A)$  and assuming that $h\lneqq 0$, 
for every $\lambda >0$, problem $(P_{\lambda})$ has  two solutions $u_{\lambda,1}, u_{\lambda,2}$ satisfying the conclusions of Theorem \ref{thm 2}.
\end{cor}

Corollary \ref{Cor Negatif}  should be compared with \cite[Theorem 3.3]{AbPePr} where the authors prove the existence only of $u_{\lambda,1}$ under  however weaker regularity assumptions.
\medbreak

Our Theorems \ref{thmlocal} - \ref{thm 2} require $(P_0)$ to have a solution and thus we are in a situation where  a branch of solutions starts from $(0,u_0)$. In our next results we consider the situation for $\lambda>0$ ``large''.

\begin{thm}
\label{thm 3}
Under assumption $(A)$ and assuming that 

\begin{enumerate}
\item[(a)] $(P_0)$ does not have a solution $u_0\leq0$;
\item[(b)] there exists $\lambda_0 > 0$ and $\beta_0$ an upper solution of  $(P_{\lambda_0})$ with $\beta_0\leq0$.% and $c \beta_0 \lneqq  0$.
\end{enumerate}
Then there exists $0 < \underline\lambda\leq\lambda_0$ such that
\begin{enumerate}
\item[(i)] for every $\lambda \in \,]\underline\lambda,+\infty[$, the problem $(P_{\lambda})$ has at least two solutions with 
$u_{\lambda,1}\ll 0$ and $u_{\lambda,1} \ll u_{\lambda,2}$.
\\ 
Moreover, if $\lambda_1< \lambda_2$, we have $u_{\lambda_1,1}\gg u_{\lambda_2,1}$;
\item[(ii)] the problem $(P_{\underline\lambda})$ has a unique solution $u_{\underline\lambda}\leq 0$;
\item[(iii)] for $\lambda<\underline\lambda$, the problem $(P_{\underline\lambda})$ has no 
solution $u\leq 0$.% with $cu\lneqq 0$.}
%\item[(iii)] for every $0<\lambda<\underline\lambda$, the problem $(P_{\lambda})$ has no solution $u$ with $u\leq 0$ and $\displaystyle \int_{\Omega} c u \varphi_1<0$. 
\end{enumerate} 
\end{thm}

\begin{figure}[t]
\begin{center}
\includegraphics[scale=0.3]{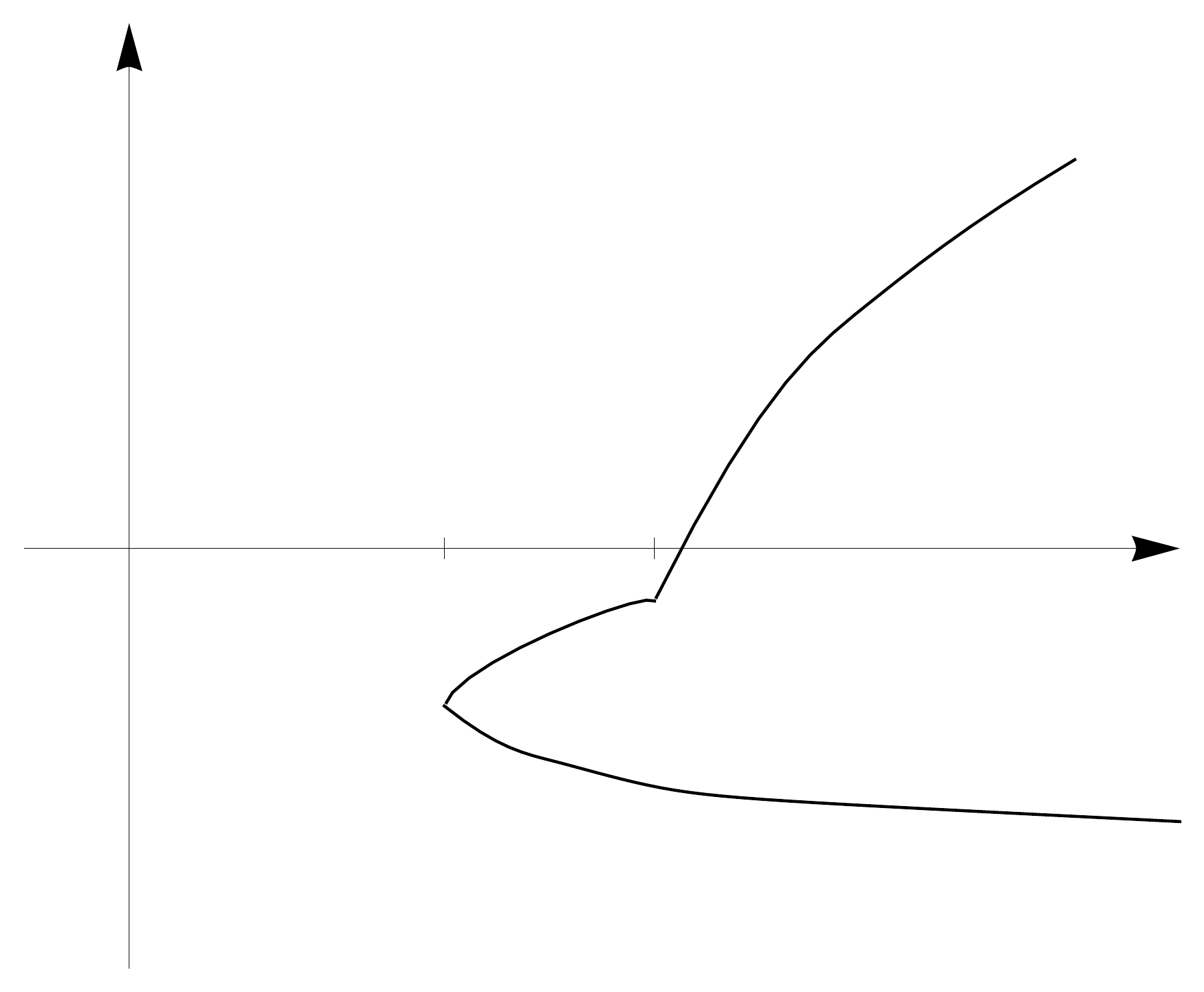}
\put(-5,50){$\lambda$}
\put(-80,65){$\lambda_0$}
\put(-105,65){$\underline{\lambda}$}
\caption{Illustration of Theorem \ref{thm 3}}
\label{figdn3}
\end{center}
\end{figure}

In our last results we change our point of view and consider no more the dependence in $\lambda$ but in $\|h^+\|$. In proving Theorem
\ref{thm 4}, we shall also obtain, in case  $\|h^+\|$ is small enough,   
the existence of a negative upper solution of   $(P_{\lambda_0})$ for some $\lambda_0 \geq 0$ as needed in the assumptions of  Theorem \ref{thm 3}.

%A consequence of Theorem \ref{thm 3} is the following result.

\begin{thm}
\label{thm 4}
Under assumption $(A)$, let $\tilde h\in L^p(\Omega)$ and consider $\tilde h^+$ and $\tilde h^-$ respectively its positive and its negative part.
Assume that $\tilde h^+\not\equiv0$. Let  $\nu_1 >0$ be the first eigenvalue of
\begin{equation}
\label{other-eigenvalue}
%\begin{array}{ccc}
-\Delta u+\mu_2 \tilde h^-(x) u  = \nu_1 c(x)u,   \quad u \in H^1_0(\Omega).
%\\
% \varphi_{1} =0,  &\mbox{ on }& \partial\Omega
%\end{array}
\end{equation}

Then, for all $\lambda>\nu_1$, there exists $\overline k=\overline k(\lambda)\in \, ]0,+\infty[$ such that, 
\begin{enumerate}
\item[(i)]
for all $k\in \,]0, \overline k[$, the problem
\begin{equation*}
- \Delta u = \lambda c(x)u+ \mu(x) |\nabla u|^2 +  k\tilde h^+(x)-\tilde h^-(x), \quad  u \in H^1_0(\Omega) \cap L^{\infty}(\Omega)
\eqno{(Q_{\lambda,k})}
\end{equation*}
has at least two solutions $u_{\lambda,1}\ll u_{\lambda,2}$;
\item[(ii)]
for all $k>\overline k$, the problem $(Q_{\lambda,k})$ has no solution;
\item[(iii)]
for $k=\overline k$, the problem $(Q_{\lambda,k})$ has  exactly one solution.
%\item[(iv)]
%we can assume that the function $\lambda \to \overline k(\lambda)$ in non decreasing.
\end{enumerate}
\end{thm}

We deduce from  Theorems \ref{thm 1} and \ref{thm 4} the following Corollary that concerns the case $h \gneqq 0$.

\begin{cor} 
\label{surprise}
Under assumption $(A)$ and assuming that $h \gneqq 0$, for all $\tilde \lambda>\gamma_1$ where $\gamma_1>0$ is the first eigenvalue \eqref{eigenvaluep},  there exists $\tilde k>0$ such that, for all $k\in \,]0,\tilde k]$, 
\begin{enumerate}
\item[(i)]
there exists $\lambda_1\in \,]0,\gamma_1[$ such that
\\
$\bullet$ for all $\lambda\in\, ]0, \lambda_1[$, the problem 
\begin{equation}
\label{eq cor}
- \Delta u = \lambda c(x)u+ \mu(x) |\nabla u|^2 + kh(x), \quad u \in H^1_0(\Omega) \cap L^{\infty}(\Omega)
\end{equation}
has at least two positive solutions;
\\
$\bullet$ for $\lambda=\lambda_1$, the problem \eqref{eq cor} has exactly one positive solution;
\\
$\bullet$ for $\lambda>\lambda_1$, the problem \eqref{eq cor} has no non-negative solution;
\item[(ii)]
for $\lambda=\gamma_1$ the problem \eqref{eq cor} has no solution;
\item[(iii)]
there exists $\lambda_2\in \,]\gamma_1, \tilde \lambda]$ such that
\\
$\bullet$  for $\lambda>\lambda_2$, the problem \eqref{eq cor} has at least two solutions with $u_{\lambda,1}\ll 0$ and $\min u_{\lambda,2}<0$;
\\
$\bullet$ for $\lambda=\lambda_2$,  the problem \eqref{eq cor} has a unique non-positive solution; 
\\
$\bullet$ for $\lambda<\lambda_2$,  the problem \eqref{eq cor} has no non-positive solution. 
\end{enumerate}
\end{cor}

\begin{remark}
Observe that, as $h\geq 0$, we have $\gamma_1=\nu_1$, where $\nu_1$ is the first eigenvalue of \eqref{other-eigenvalue} and $\gamma_1$ is the first eigenvalue of 
\eqref{eigenvaluep}.
\end{remark}

\begin{figure}[t]
\begin{center}
\includegraphics[scale=0.3]{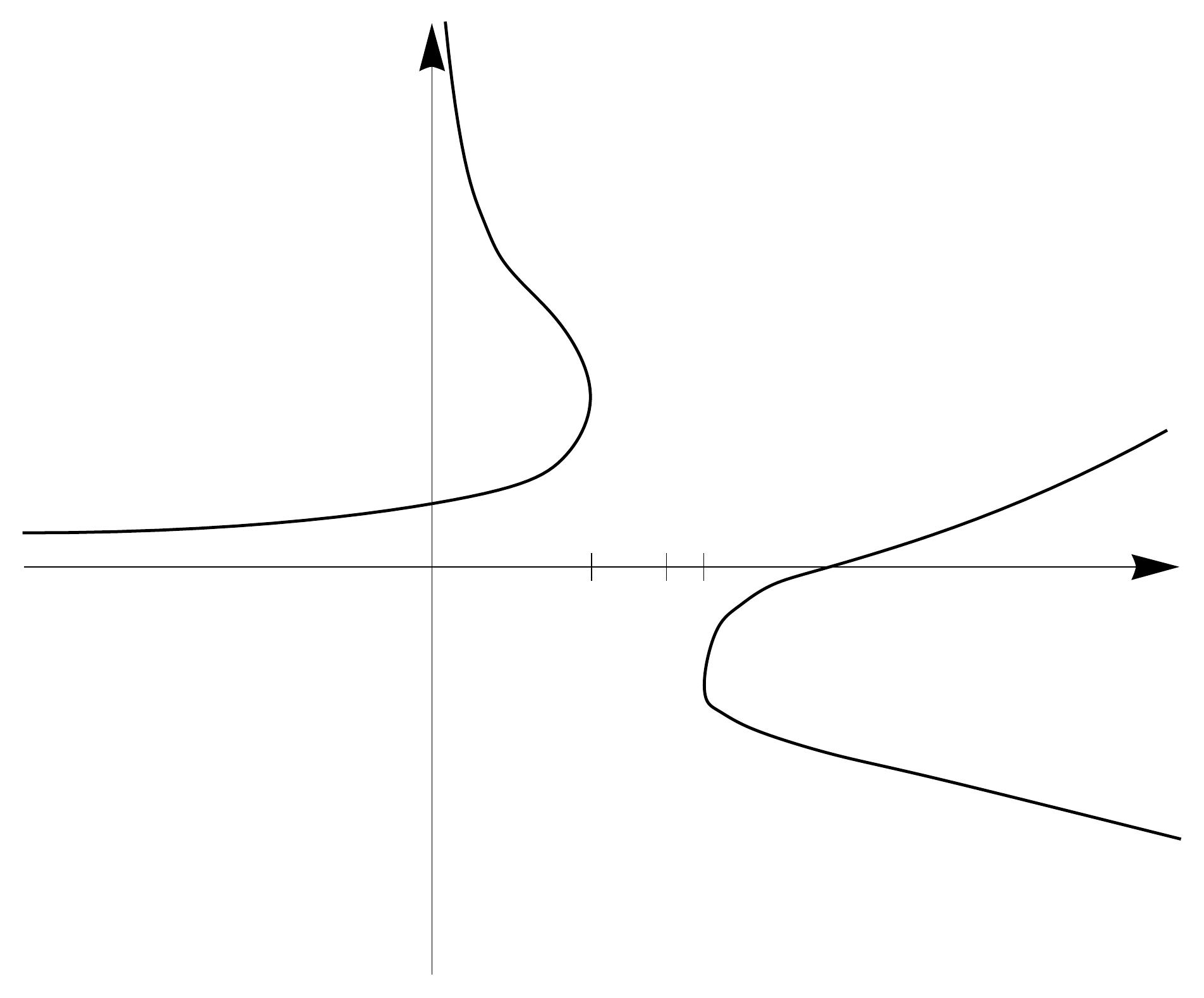}
\put(-0,60){$\lambda$}
\put(-90,49){$\lambda_1$}
\put(-78,49){$\gamma_1$}
\put(-70,65){$\lambda_2$}
\caption{Illustration of Corollary \ref{surprise}}
\label{figdn4}
\end{center}
\end{figure}
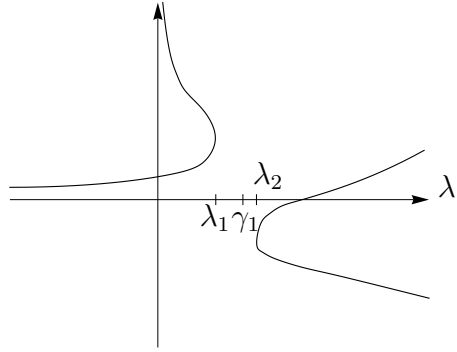

We conclude this paper considering the case $h\equiv 0$ which can be seen  as intermediate between  the case $h \gneqq 0$ considered in Corollary \ref{surprise} and the case $h \lneqq 0$ considered in Corollary \ref{Cor Negatif}. Observe also that if we consider the problem \eqref{eq cor} with $k\in \,]-\infty, \tilde k]$, then, it is easy to see that the lower of the two solutions tends to $0$ and that $\lambda_1\to \gamma_1$, $\lambda_2\to \gamma_1$ as $k\to 0$.

\begin{thm}
\label{cas h0}
Under assumption $(A)$ with $h \equiv 0$ and recalling that $\gamma_1>0$ denotes the first eigenvalue \eqref{eigenvaluep},   we have
\begin{enumerate}
\item[(i)]
for all $\lambda\in\, ]0, \gamma_1[$, the problem 
\begin{equation}
\label{eq 0}
- \Delta u = \lambda c(x)u+ \mu(x) |\nabla u|^2, \quad u \in H^1_0(\Omega) \cap L^{\infty}(\Omega)
\end{equation}
has at least two solutions $u_{\lambda,1}\equiv 0$ and $u_{\lambda,2}\gneqq 0$ 
 with $\displaystyle\max_{\overline\Omega} u_{\lambda,2}\to + \infty$  as $\lambda \to 0$;
\item[(ii)]
for $\lambda=\gamma_1$ the problem \eqref{eq 0} has only the trivial solution;
\item[(iii)]
for $\lambda>\gamma_1$, the problem \eqref{eq 0} has at least two solutions $u_{\lambda,1}\equiv0$ and $u_{\lambda,2}\ll 0$. 
\end{enumerate}
\end{thm}

\begin{figure}[t]
\begin{center}
\includegraphics[scale=0.3]{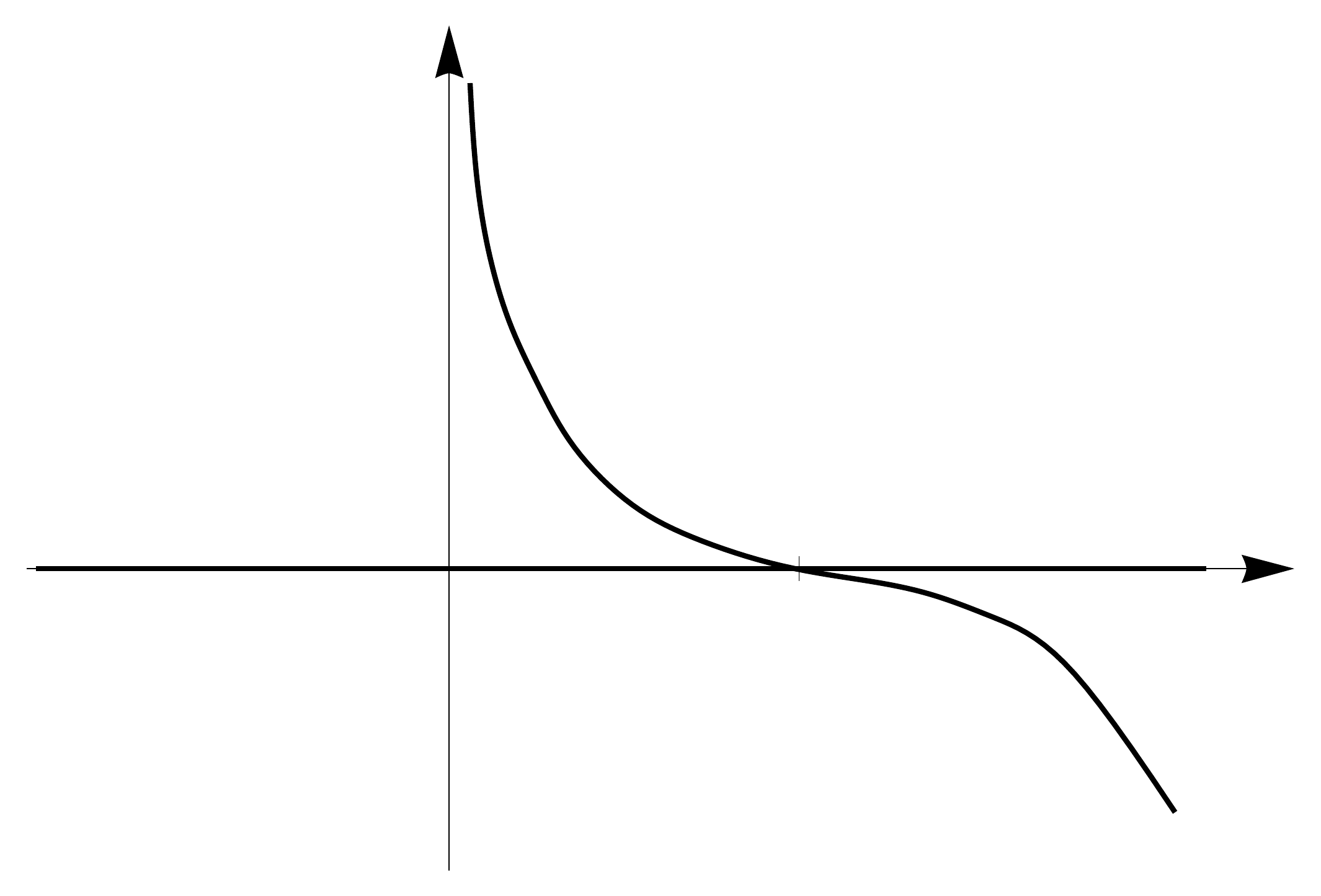}
\put(-10,53){$\lambda$}
%\put(-90,49){$\lambda_1$}
\put(-78,53){$\gamma_1$}
%\put(-70,65){$\lambda_2$}
\caption{Illustration of Theorem \ref{cas h0}}
\label{figdn4bis}
\end{center}
\end{figure}

\begin{remark}
\label{Parabolique}
Considering the solutions of $(P_{\lambda})$ as stationary solutions for the corresponding parabolic problem, assuming $(A)$ together with $\partial\Omega$ is of class $C^2$ and $c$, $h\in L^p(\Omega)$ with $p>N+2$, then, applying \cite[Corollary 2.34 and Proposition 2.41]{DeOm}, we can prove that, in the above results, the first solution $u_{\lambda,1}$ is ${\mathcal L}$-asymptotically stable from below and  $u_{\lambda,2}$ is ${\mathcal L}$-unstable from below. In the particular case of  Theorem \ref{thm 2}, 
as  $(P_{\lambda})$ has a unique negative solution $u_{\lambda,1}\ll u_0$, we have also $u_{\lambda,1}$ is ${\mathcal L}$-asymptotically stable. Fore more informations, see \cite{DeOm}. 
\end{remark}

Our existence results relies on the obtention  of a priori bounds on the solutions, see Lemma \ref{lem 4.1} and Theorem \ref{bounds1}. 
These results which are valid for arbitrary solutions use in a central way the assumption that $\mu(x) \geq \mu_1 >0$ for some $\mu_1 >0$. 
Removing this condition seems delicate and in that direction some results are obtained in \cite{So} for non-negative solutions.  
In \cite{So} it is also shown that some conditions are necessary to obtain a priori bounds for non-negative solutions. \medskip

In the case $\mu >0$ constant it is possible to precise the blow-up rate, as $\lambda \to 0^+$, of our solutions $u_{\lambda, 2}$ 
obtained in Theorems \ref{thmlocal}, \ref{thm 1},  \ref{thm 2} and \ref{cas h0}. As a by-product, we also obtain that the a priori 
estimates obtained in Theorem \ref{bounds1} are sharp. 
\medskip

The paper is organized as follows.  In Section \ref{Section1} we present some  preliminary results. Section \ref{Section2} is devoted to our a priori bounds results. 
In Section \ref{Section3} we prove our main results.  Section \ref{Section4} is devoted to the special case $\mu$ constant and in Section \ref{Section5} the reader can find  a list of open problems.
\medbreak

\noindent
{\bf Acknowledgments} The authors thank D. Mercier and C. Troestler for fruitful discussions on the interpretation of the results and for providing them the figures of the paper. The authors also thank warmly B. Sirakov for pointing to them a mistake in an earlier version of this work.
\smallskip

\noindent
{\bf Notations} For $v \in H_0^1(\Omega)$ we set $v^+ = \max\{0,v\}$ and $v^- = \max\{0, -v\}.$

\section{Preliminary results}
\label{Section1}

In our proofs we shall need some results on lower and upper solutions that we present here adapted to our setting. We consider the problem
\begin{equation}
\label{De-He}
\begin{array}{cl}
- \Delta u = f(x,u, \nabla u), &\mbox{in }\Omega
\\
u=0,&\mbox{on } \partial\Omega,
\end{array}
\end{equation}
where $f$ is an $L^p$-Carath\' eodory function with $p >N$ and solutions are sought in $W_0^{2,p}(\Omega)$.  We recall   that a {\it regular lower solution} (respectively a {\it regular upper solution}) 
of \eqref{De-He} is a function $\alpha$ (resp. $\beta$) in $W^{2,p}(\Omega)$ such that
$$
\begin{array}{cl}
- \Delta \alpha(x) \leq f(x,\alpha(x), \nabla \alpha(x)),& \mbox{for a.e. } x\in \Omega, 
\\
\alpha(x) \leq 0,& \mbox{for all } x\in \partial\Omega, 
\end{array}
$$
(respectively 
$$
\begin{array}{cl}
 - \Delta \beta(x) \geq f(x, \beta(x), \nabla \beta(x)), & \mbox{for a.e. } x\in \Omega,
\\
\beta(x) \geq 0,&\mbox{for all } x\in \partial\Omega).
\end{array}
$$
We define a {\it lower solution $\alpha$ of \eqref{De-He}} as $\alpha := \max\{\alpha_i\mid 1\leq i \leq k\}$ where $\alpha_1, \ldots,  \alpha_k$ are regular lower solutions of \eqref{De-He}.
Similarly,  an {\it upper solution $\beta$ of \eqref{De-He}} is defined as  $\beta=\min\{\beta_j\mid 1\leq j \leq l\}$ where $\beta_1, \ldots,  \beta_l$ are regular upper solutions of \eqref{De-He}.

\begin{remark} 
The set of functions $w$ such that $u \ll w \ll v$ is open
in
$C^1_0(\Omega)$
(the space  of the $C^1$-functions in $\overline\Omega$ which vanish on the boundary of $\Omega$).
\end{remark}

Problem \eqref{De-He} can be transformed into a fixed point problem. The operator
\begin{equation}
\label{defL}
{\mathcal L} : W^{2,p}_0(\Omega) \to L^p(\Omega) ;\,
u \mapsto -\Delta u 
\end{equation}
is a linear homeomorphism.

Since $f$ is an  $L^p$-Carath\'eodory function,  the operator
\begin{equation}
\label{defN}
{\mathcal N} : C_0^1(\overline{\Omega}) \to L^p(\Omega)\, ;\, u \mapsto f(.,u(.),\nabla u(.))
\end{equation}
is well defined, continuous and maps bounded sets to bounded sets.  Since
$p>N$,  as $W^{2,p}_0(\Omega)$ is compactly embedded in
$C^1_0(\overline \Omega)$, the
operator ${\mathcal M} : C^1_0(\overline\Omega) \to C^1_0(\overline\Omega)$
defined by
$$
 {\mathcal M} (u) = {\mathcal L}^{-1} {\mathcal N} u,
$$
where ${\mathcal L}$ and ${\mathcal N}$ are given respectively by \eqref{defL} and \eqref{defN}, is completely continuous and
the
problem \eqref{De-He} is equivalent
to
$$
 u = {\mathcal M} u.
$$

To be able to associate a degree to a pair of lower and upper solutions we also need to reinforce the definition.

\begin{definition}
A lower solution $\alpha$ of \eqref{De-He} is said to be {\it strict} if
every solution $u$ of \eqref{De-He} such that $\alpha\leq u$ on $\Omega$
satisfies $\alpha\ll u$.

In the same way a {\it strict upper solution} $\beta$ of \eqref{De-He} is an
upper solution such that every solution $u$ with $u\leq\beta$ is such that
$u \ll \beta$.
\end{definition}

Our main tool regarding the existence and characterizations of solutions of problem \eqref{De-He} by a lower and upper solutions approach is the following theorem. This result which can be obtained adapting some ideas from \cite{DeHa, DeOm} will be proved in the Appendix.

\begin{thm}
%\cite[Theorem 3.1]{DeHe}
\label{sousDe}
Let $\Omega$ is a bounded domain in ${\mathbb R}^N$  with boundary $\partial \Omega$
of class $C^{1,1}$ and  $f$ be an
$L^p$-Carath\'eodory function with $p>N$.
  Assume that there exists a lower solution  $\alpha$ and an upper solution $\beta$  of \eqref{De-He} such that $\alpha \leq \beta$. Denote 
  $\alpha := \max\{\alpha_i\mid 1\leq i \leq k\}$ where $\alpha_1, \ldots,  \alpha_k$ are regular lower solutions of \eqref{De-He} and 
  $\beta=\min\{\beta_j\mid 1\leq j \leq l\}$ where $\beta_1, \ldots,  \beta_l$ are regular upper solutions of \eqref{De-He}.
    If there exists $K>0$ and $h \in L^p(\Omega)$
such that for a.e.\ $x \in \Omega$, all $u \in [\min\{\alpha_i\mid 1\leq i \leq k\},\max\{\beta_j\mid 1\leq j \leq l\}]$ and
all $\xi \in {\mathbb R}^N$,
\begin{equation}
 \label{Hf}
|f(x,u,\xi)| \leq h(x) + K |\xi|^2,
\end{equation}
then the problem \eqref{De-He} has at least one solution $u$
satisfying 
$$
\alpha \leq u  \leq \beta.
$$
Moreover, problem \eqref{De-He} has a minimal solution $u_{\min}$ and a maximal solution $u_{\max}$ in the sense that, $u_{\min}$ and $u_{\max}$ are solutions of \eqref{De-He} with $\alpha\leq u_{\min}  \leq u_{\max} \leq \beta$ and 
every solution $u$ of \eqref{De-He} with $\alpha \leq u \leq \beta$ satisfies $u_{\min} \leq u  \leq u_{\max}$.

If moreover $\alpha$ and $\beta$ are strict and satisfy $\alpha \ll
\beta$, then, there exists $R>0$ such that
$$
\deg(I-{\mathcal M}, {\mathcal S}) = 1,
$$
where
$$
{\mathcal S} =
\{u \in C^1_0(\overline\Omega) \mid \alpha \ll u \ll \beta, \quad \|u\|_{C^1}<R \}.
$$
\end{thm}

\begin{remark}\label{rem-utile}
If $\alpha$ and $\beta$ are respectively strict lower and upper solutions of \eqref{De-He} with $\alpha\leq \beta$ then $\alpha\ll\beta$. 
Indeed from the first part of Theorem \ref{sousDe} we deduce, the existence of a solution $u$ with  $\alpha\leq u \leq \beta$. 
By definition of strict lower and upper solutions, we obtain $\alpha\ll u \ll\beta$ and hence $\alpha\ll\beta$.
\end{remark}

\begin{remark}
We shall apply Theorem \ref{sousDe} with
${\mathcal N}(u) = \lambda c(x) u + \mu(x) |\nabla u|^2+ h(x)$.
Hence, as we are concerned with the $\lambda$-dependance, we will denote the fixed point operator ${\mathcal M}_{\lambda}$ instead of ${\mathcal M}$.
\end{remark}

Our assumption (A) implies that the following regularity result applies to problem $(P_{\lambda})$. 

\begin{thm}
\label{regularity}
Let $\Omega$ is a bounded domain in ${\mathbb R}^N$ with boundary $\partial \Omega$
of class ${C}^{1,1}$, $c\in L^p(\Omega)$, $h\in L^p(\Omega)$ with $p>N$ and $\mu\in L^{\infty}(\Omega)$.
Let $u$ be a solution of
\begin{equation}
\label{1111}
- \Delta u = c(x)u+\mu(x)|\nabla u|^2+h(x), \quad u \in H^1_0(\Omega) \cap L^{\infty}(\Omega).
\end{equation}
Then $u\in W_0^{2,p}(\Omega) \subset {C}_0^1(\overline\Omega)$.
\end{thm}

\begin{remark}
This result is not a simple consequence of classical bootstrap arguments as, for $u\in H^1_0(\Omega)\cap L^{\infty}(\Omega)$, $\mu |\nabla u|^2\in L^1(\Omega)$ which does not allow to start a bootstrap process.
\end{remark}

\begin{remark}
Observe that any solution $u \in C_0^1(\overline\Omega)$ of \eqref{1111} belongs to $W_0^{2,p}(\Omega)$.
\end{remark}

\begin{proof}
Let $u\in H^1_0(\Omega) \cap L^{\infty}(\Omega)$ and define the function $g= c\,u+u+h$. Observe that $g\in L^p(\Omega)$ with $p>N$ and $u$ is solution of 
\begin{equation}
\label{2222}
- \Delta v = - v + \mu(x)|\nabla v|^2 + g(x), \quad v \in H^1_0(\Omega) \cap L^{\infty}(\Omega).
\end{equation}

Let us prove that this problem has a solution $v\in W_0^{2,p}(\Omega).$
By uniqueness of the solution of \eqref{2222} in 
$H^1_0(\Omega) \cap L^{\infty}(\Omega)$ (see \cite[Theorem 1.1]{ArDeJeTa2}), we obtain that $v=u$ and hence $u\in W_0^{2,p}(\Omega)$.
To prove that this problem has a solution $v\in W_0^{2,p}(\Omega)$,  we shall apply Theorem \ref{sousDe}. Thus we need to prove that \eqref{2222} has a lower $\alpha$ and an upper solution $\beta$ with $\alpha\leq \beta$. \medskip

We set $\overline{\mu} = ||\mu||_{\infty}$. Clearly any solution of 
\begin{equation}
\label{1}
\begin{array}{cl}
- \Delta u = -u + \overline{\mu} |\nabla u|^2 + g^+(x), & \mbox{ in } \Omega,
\\
u=0, &\mbox{ on } \partial \Omega,
\end{array}
\end{equation}
is an upper solution of \eqref{2222} and   
any solution of
\begin{equation}
\label{222}
\begin{array}{cl}
- \Delta u = -u - \overline{\mu} |\nabla u|^2 - g^-(x),  & \mbox{ in } \Omega,
\\
u=0, &\mbox{ on } \partial \Omega,
\end{array}
\end{equation}
is a lower solution of \eqref{2222}. 
Now if $w \in W^{2,p}_0(\Omega) $ is a solution of 
\begin{equation}
\label{3b}
\begin{array}{cl}
- \Delta w = -w + \overline{\mu} |\nabla w|^2 + g^-(x),  & \mbox{ in } \Omega,
\\
w=0, &\mbox{ on } \partial \Omega,
\end{array}
\end{equation}
then $u = - w$ satisfies (\ref{222}). 
Thus if we find a non-negative solution $u_1 \in W^{2,p}_0(\Omega)$ of \eqref{1} and a non-negative solution $u_2 \in W^{2,p}_0(\Omega)$ 
of \eqref{3b} then, setting $\beta = u_1$ and $\alpha = -u_2$, we have the required couple of lower and upper solutions required to apply Theorem \ref{sousDe}. 

Let us construct $u_1$, the construction of $u_2$ being similar.
Let
$w_1 \in H^1_0(\Omega)$ be the non-negative solution of
$$
\begin{array}{cl}
-\Delta w_1 = \overline{\mu} g^+(x) w_1-m(w_1)+g^+, &\mbox{ in } \Omega,
\\
w_1=0,  &\mbox{ on } \partial \Omega
\end{array}
$$
where
\begin{equation}
\label{defg}
\begin{array}{ll}
m(s) = 
\left\{
\begin{array}{ll} 
\text{$\frac{1}{\overline{\mu}} (1+ \overline{\mu} s) \ln (1+ \overline{\mu} s)$},   &\mbox{if } \quad s \geq 0,
\\
\text{$-\frac{1}{\overline{\mu}} (1-\overline{\mu} s) \ln (1- \overline{\mu} s)$},   &\mbox{if } \quad s < 0
\end{array}
\right.
\end{array}
\end{equation}
given by \cite[Lemma 3.3]{ArDeJeTa}. By \cite[Lemma  3.22]{Tr87} and a bootstrap argument, it is easy to prove that $w_1\in W^{2,p}(\Omega)$.
Hence 
$$ 
u_1=\frac{\ln(\overline{\mu} w_1 +1)}{\overline{\mu}}
\in W^{2,p}(\Omega)
$$
and one readily shows that  $u_1\geq 0$ is a solution of 
\eqref{1}. 
\end{proof}

\begin{prop}
\label{comparison}
Under assumption (A) if $\alpha$ is a lower solution of $(P_{0})$ and $\beta$ an upper solution of $(P_0)$ then $\alpha \leq \beta$.
\end{prop}

\begin{proof}
Since $\alpha=\max\{\alpha_i\mid 1\leq i \leq k\}$, $\beta=\min\{\beta_j\mid 1\leq j \leq l\}$ with $\alpha_i$ and $\beta_j$ regular lower and upper solutions in  
$W^{2,p}(\Omega)$, the functions $\alpha_i$ and $\beta_j$ belong to  $H^1(\Omega) \cap W_{loc}^{1,N}(\Omega) \cap C(\overline \Omega) $ 
and we conclude by \cite[Lemma 2.2]{ArDeJeTa2}.
\end{proof}

The following estimates will also be useful.

\begin{lem}[Nagumo Lemma]
\label{lem 2}
Let $p>N$, $h\in L^p(\Omega, \mathbb R^+)$, $K>0$, $R>0$. Then there exists $C>0$ such that, for all $u\in W^{2,p}(\Omega)$ satisfying
$$
\begin{array}{cl}
|\Delta u|\leq h(x)+K|\nabla u|^2, &\mbox{a.e. in }\Omega,
\\
u=0,  &\mbox{on }\partial \Omega,
\end{array}
$$
and
$$
\|u\|_{\infty}\leq R,
$$
we have
$$
\|u\|_{W^{2,p}}\leq C.
$$
\end{lem}

\begin{proof} see \cite[Lemma 5.10]{Tr87}.
\end{proof}

\begin{lem}
\label{Sirakov}
Assume that $c, h \in L^q(\Omega)$ for some $q > \frac{N}{2}$. Then if $u \in H^1_0(\Omega)$ is solution of
$$ 
- \Delta u \leq c(x) u + h(x), \qquad (\mbox{resp. }   - \Delta u \geq c(x) u + h(x))
$$
in a weak sense,
then $u$ is bounded above (resp. below) and
$$
\sup_{\Omega}u^+  \leq C ( \|u^+ \|_2 + \|h\|_q), 
\qquad
 (\mbox{resp. }   
\sup_{\Omega} u^- \leq C ( \|u^-\|_2 + \|h\|_q)),
$$
where $C>0$ depends on $N, q, |\Omega|$ and $\|c\|_q$.
\end{lem}

\begin{proof} see \cite[Lemma 5]{JeSi}.
\end{proof}

We also need the following formulation of the anti-maximum principle. Under slightly more smooth data this result was established in \cite{PeHe} but the proof given in \cite{PeHe} directly extend under our regularity assumptions. 

\begin{prop}
\label{anti}
Let $\bar{c}, \bar{h}, \bar{d} \in L^p(\Omega)$ with $p >N$ and assume that $\bar{h} \gneqq 0$. We denote by $\bar{\nu}_1 >0$ the first eigenvalue of
\begin{equation}
\label{cl}
-\Delta u + \bar{d}(x) u = \bar{\nu}_1 \bar{c}(x) u, \quad u \in H^1_0(\Omega).
\end{equation}
Then there exists $\varepsilon_0>0$ such that, for all $\lambda\in\,]\bar{\nu}_1,\bar{\nu}_1+\varepsilon_0]$, the solution $w$ of
\begin{equation}
\label{cll}
-\Delta w+ \bar{d}(x)w = \lambda \bar{c}(x) w  + \bar{h},\quad u \in H^1_0(\Omega).
\end{equation}
satisfies $w\ll0$.
\end{prop}

\section{A priori bound}
\label{Section2}

This section is devoted to the derivation of some a priori bounds results for the solutions of $(P_{\lambda})$. Most of our results hold under more general assumptions than (A). \medskip

First, using ideas of \cite{AbPePr}, we obtain the following lower bound on the upper solutions of $(P_{\lambda})$.  %Here we assume \smallskip
%$$
%\hspace{1cm}
%\left\{ \begin{array}{c} 
%\Omega \subset \R^N,\,\,  N \geq 2 \,\,\mbox{ is a bounded domain with } \partial \Omega
%\mbox{  of class } C^{1,1},
%\\[2mm]
%c \mbox{ and } h \mbox{ belong to }  L^p(\Omega) \,\, \mbox{for some } p > N/2,
%\\[2mm]
% \mu \in L^{\infty}(\Omega) \mbox{ satisfies } 0<\mu_1\leq \mu(x)\leq \mu_2.
%\end{array}
%\right.
%\leqno{\mathbf{(B)}}
%$$
%\medbreak

\begin{lem}
\label{lem 4.1}
Under conditions $(A)$, for any $\Lambda_2>0$, there exists a constant $M:=M(\Lambda_2, \mu_1, \|c\|_{N/2}, \|h^-\|_{N/2})>0$ such that, 
for any $\lambda \in [0, \Lambda_2]$,  any function  $u\in H^1_0(\Omega) \cap L^{\infty}(\Omega)$ verifying  $u\geq 0$ on $\partial \Omega$ and such that, for all 
$v \in H^1_0(\Omega) \cap L^{\infty}(\Omega)$ with $v\geq 0$ a.e. in $\Omega$,
\begin{equation}
\label{EE1bis}
\int_{\Omega} \nabla u \nabla v \, dx 
\geq  
\int_{\Omega} [\lambda c(x)u + \mu(x) |\nabla u|^2 + h(x)] v\, dx
\end{equation}
satisfies
$$
\min_{\Omega} u >-M.
$$
\end{lem}

%\textcolor{red}{ $u\in H^1_0$ ou $u\geq 0$ sur $\partial\Omega$ ??}

\begin{remark}
This result is valid under less regularity conditions than (A) and without sign conditions on $c$ and $h$.
More precisely, it holds under the conditions:
$\Omega \subset \R^N$, $N \geq 2$ is a bounded domain with $\partial \Omega$
of class $C^{1,1}$,
$c$ and  $h$  belong to $L^p(\Omega)$ for some $p > N/2$,
$\mu \in L^{\infty}(\Omega)$ satisfies $0<\mu_1\leq \mu(x)\leq \mu_2$.
 
Moreover, the lower bound does not depend on $h^+$ and depends only on an upper bound on $\lambda \geq 0$. 
\end{remark}

\begin{proof}
Let us take $v=u^-$ as test function in  \eqref{EE1bis}. We obtain
$$
\begin{array}{rcl}
\displaystyle 
-\int_{\Omega}|\nabla u^-|^2 \, dx
&\geq& 
\displaystyle
- \lambda \int_{\Omega} c^+ (u^-)^2 \, dx + \mu_1 \int_{\Omega} |\nabla u^-|^2u^- \, dx- \int_{\Omega} h^- u^-\, dx
\\[3mm]
& \geq & 
\displaystyle
- \Lambda_2 \int_{\Omega} c^+ (u^-)^2 \, dx+ \mu_1\frac49 \int_{\Omega} |\nabla (u^-)^{3/2}|^2 \, dx - \int_{\Omega} h^- u^-\, dx
\end{array}
$$
and hence
$$
 \mu_1\frac49 \int_{\Omega} |\nabla (u^-)^{3/2}|^2 \, dx + \int_{\Omega}|\nabla u^-|^2\, dx 
\leq
\int_{\Omega} h^- u^- \, dx + \Lambda_2 \int_{\Omega} c^+ (u^-)^2 \, dx.
$$
For every $\varepsilon>0$ we have
$$
\begin{array}{rcl}
\displaystyle
\Lambda_2 \int_{\Omega} c^+ (u^-)^2 \, dx
&=&
 \displaystyle
\int_{\Omega} (\Lambda_2 c^+)^{1/2} (u^-)^{1/2} (\Lambda_2 c^+)^{1/2} (u^-)^{3/2} \, dx
\\[3mm]
&\leq& 
\displaystyle
\frac{1}{2\varepsilon} \Lambda_2 \int_{\Omega} c^+ u^- \, dx+ \frac{\varepsilon}{2} \Lambda_2 \int_{\Omega} c^+ \left((u^-)^{3/2}\right)^2 \, dx.
\end{array}
$$
Also, for some constant $C_N$, by Sobolev's embedding, we get
$$
\int_{\Omega} c^+ \left((u^-)^{3/2}\right)^2 \, dx
\leq 
\|c^+\|_{N/2} \|(u^-)^{3/2}\|_{2^*}^2 
\leq 
\frac{1}{C_N} \|c^+\|_{N/2} \|\nabla (u^-)^{3/2}\|_{2}^2.
$$
We then obtain
$$
\begin{array}{l}
\displaystyle
 \mu_1\frac49 \int_{\Omega} |\nabla (u^-)^{3/2}|^2 \, dx
+
\displaystyle\int_{\Omega}|\nabla u^-|^2 \, dx
\\[3mm]
\mbox{}\hspace{20mm}
\leq
\displaystyle
\int_{\Omega} h^- u^- \, dx+ \frac{1}{2\varepsilon} \Lambda_2 \int_{\Omega} c^+ u^- \, dx
+ \frac{\varepsilon}{2} \frac{\Lambda_2}{C_N} \|c^+\|_{N/2} \|\nabla (u^-)^{3/2}\|_{2}^2.
\end{array}
$$
Hence, by choosing $\displaystyle \varepsilon=\frac{C_N}{\Lambda_2  \|c^+\|_{N/2}} \mu_1 \frac49$, it comes
\arraycolsep1.5pt
$$
\begin{array}{rcl}
\displaystyle
 \mu_1\frac29 \int_{\Omega} |\nabla (u^-)^{3/2}|^2 \, dx
&+& 
\displaystyle
\int_{\Omega}|\nabla u^-|^2 \, dx
\\
&\leq&
\displaystyle
\int_{\Omega} h^- u^- \, dx+ \frac{9 \Lambda_2^2  \|c^+\|_{N/2}}{8 \mu_1 C_N} \int_{\Omega} c^+ u^-\, dx 
\\[3mm]
&\leq&
\displaystyle
 C \! \left(\|h^-\|_{N/2} \|\nabla u^-\|_2+\frac{\Lambda_2^2}{\mu_1} \|c^+\|_{N/2}^2  \|\nabla u^-\|_2\right)
\end{array}
$$
\arraycolsep5pt
from which we deduce that
$$
\|u^-\|_{H^1_0} \leq C (\|h^-\|_{N/2} +\frac{\Lambda_2^2}{\mu_1}\|c^+\|_{N/2}^2).
$$
By Lemma \ref{Sirakov} we obtain that
$$
u\geq -M:=-M(\Lambda_2, \mu_1, \|c\|_{N/2}, \|h^-\|_{N/2})
$$
which allows to conclude.
\end{proof}

As a simple corollary we have the following result.

\begin{cor}
\label{cor 4.1b}
Under conditions $(A)$, for any $\Lambda_2>0$, there exists a constant $M:=M(\Lambda_2, \mu_1, \|c\|_{N/2}, \|h^-\|_{N/2})>0$ such that, 
for any $\lambda \in [0, \Lambda_2]$,  any upper solution $\beta$ of $(P_{\lambda})$
satisfies
$$
\min_{\Omega} \beta >-M.
$$
\end{cor}

\begin{proof}
As $\beta=\min\{\beta_j\mid 1\leq j \leq l\}$ where $\beta_j$ are regular  upper solutions, they belong to  $H^1(\Omega) \cap L^{\infty}(\Omega) $ and 
satisfy \eqref{EE1bis}. We conclude by Lemma \ref{lem 4.1}.
\end{proof}

Let $\tilde\nu_1 >0$ denotes the first eigenvalue of
\begin{equation}
\label{eigenvaluep2}
-\Delta u + \mu_1 h^-(x) u= \nu c (x) u,  \quad u \in H^1_0(\Omega),
\end{equation}
with corresponding eigenfunction $\psi_1 >0$. 
\medbreak

\begin{thm}
\label{bounds1}
Under condition $(A)$, for any $\Lambda_2>\Lambda_1>0$, any $A>0$, there exists a constant $M>0$ such that,
 for any $\lambda\in [\Lambda_1,\Lambda_2]$, $a\in [0,A]$, any solution $u$ of
 \begin{equation}
\label{EE1a}
- \Delta u = \lambda c(x)u+ \mu(x) |\nabla u|^2 + h(x)+ac(x), \quad u \in H^1_0(\Omega) \cap L^{\infty}(\Omega),
\end{equation}
 satisfies
$$
\|u\|_{\infty}<M.
$$
Moreover, viewed as a function of $\Lambda_1$,  $M = O_{0^+}(1/\Lambda_1)$. %\smallskip
\end{thm}

In the above theorem, the notation $M = O_{0^+}(g(\Lambda_1))$ means the existence of $C>0$ such that
$$
\left|\frac{M(\Lambda_1)}{g(\Lambda_1)}\right|\leq C, \qquad \mbox{ as }\Lambda_1\to 0^+.
$$

\begin{remark}
The above theorem is valid under less restrictive conditions. In fact it is valid if
we replace the regularity $c$ and $h\in L^p(\Omega)$ with $p>N$ by  $c$ and $h\in L^p(\Omega)$ with $p>N/2$ and $h^- \in L^q(\Omega)$ for some $q >N$. This last condition is used to prove that the first eigenfunction $\psi_1 >0$ of \eqref{eigenvaluep2} satisfies $\psi_1 \geq d \delta(x)$ for some constant  $d>0$ where $\delta(x)$
denotes the distance from $x$ to $\partial \Omega$. This is needed  to insure that the conclusion of Lemma \ref{tec2} holds.
Following the proof of \cite[Lemma 6.3]{ArDeJeTa} it is  possible to prove that this condition on $\psi_1$ holds under this stronger regularity.
\end{remark}

In the proof of the Theorem \ref{bounds1} the following  technical lemmas will be used.

\begin{lem}
\label{tec1}
Let $p > \frac{N}{2}$ and $\theta \in\, ]0,1[$. There exist   $r \in\, ]0,1[$ and  $\alpha \in ]0, \frac{p-1}{2p-1}[$ such that if  we define
    \begin{equation}
		\label{a0}
    q =1+ r + \frac{1+ \theta \alpha}{1- \alpha}, \quad \tau = \frac{1}{q} \, \frac{\alpha}{1- \alpha}
    \end{equation}
then it holds
\begin{equation}
\label{boundsq}
\frac{1}{p} \leq q \leq \frac{2N(p-1)}{p(N-2 + 2 \tau)}
\end{equation}
and
\begin{equation}
\label{boundsalpha}
1 - \alpha < \frac{2}{q}.
\end{equation}
\end{lem}

\begin{proof}
See \cite[Lemma 6.2]{ArDeJeTa}.
\end{proof}

\begin{lem}
\label{tec2}
Let $b \in L^p(\Omega)$ with $p > \frac{N}{2}$. For any $p$, $q\geq 1$ and $\tau\in [0,1]$ satisfying \eqref{boundsq}, there exists $C>0$ such that,
for all $w\in H^1_0(\Omega)$,
    $$  
		\left\|\frac{b^{1/q} w}{\psi_1^\tau}\right\|_q \leq C\| b\|_p \|\nabla w\|_2,
    $$
		where $\psi_1 >0 $ denotes the first eigenfunction  \eqref{eigenvaluep2}. 
		% the first eigenfunction of 
		%\begin{equation}
%\label{eigenvaluep}
%\begin{array}{ccc}
%-\Delta \varphi_{1} = \gamma c (x) \varphi_{1},  \quad \varphi_1 \in H^1_0(\Omega),
%\\
% \varphi_{1} =0,  &\mbox{ on }& \partial\Omega
%\end{array}
%\end{equation}

\end{lem}

\begin{proof}
See \cite[Lemma 6.3]{ArDeJeTa} or \cite{BT77}.
\end{proof}

%\begin{lem}
%\label{lem a}
%Assume that $c$, $h\in L^p(\Omega)$ for some $p>N/2$ and $\mu\in L^{\infty}(\Omega)$ are such that $c\gneqq 0$, $0<\mu_1\leq\mu(x)\leq\mu_2$ and $\lambda>0$. Then %for any $A_1>0$ and for %any $\varepsilon>0$, 
%there exists a constant $C >0$ such that, 
%%for each $a\in [0,A_1]$, 
%for any $w\in H^1_0(\Omega)\cap L^{\infty}(\Omega)$ with $w\geq -1/\mu_1$, we have
%$$
%\begin{array}{c}
%\displaystyle
%\int_{\Omega} c(x)w\psi_1 \leq \frac{1}{2\gamma_1} \int_\Omega (1+\mu_1w)\left[\lambda c(x)g_1(w) +h(x)\right]\varphi_1\, dx + C,
%\\
%\displaystyle
%\int_{\Omega} (1+\mu_1 w)(\lambda c(x) g_1(w)+h(x)) w_1^-
%\geq -C \int_{\Omega} \lambda c(x)   w_1^-,
%\end{array}
%$$
%and,  for all $w\in H^1_0(\Omega)\cap L^{\infty}(\Omega)$ with $w\geq -1/\mu_2$, we have
%$$
%\int_\Omega (1+\mu_2w)^{1-\theta}(c(x)\lambda g_2(w)+h(x))^-\varphi_1\,dx\leq C,
%$$
%with 
%$\theta=(\mu_2-\mu_1)/ \mu_2\in \,]0,1[$.
%\end{lem}

%\end{proof}

\begin{proof}[Proof of Theorem \ref{bounds1}]
Let  $\lambda \in [\Lambda_1, \Lambda_2]$, $a\in [0,A]$ and $u$ be a solution of \eqref{EE1a}. Assume without loss of generality that $\Lambda_1\leq 1\leq \Lambda_2$.
We define
    $$  
    w_i(x)=\frac{1}{\mu_i}(e^{\mu_iu(x)}-1) \, \mbox{ and } \,  g_i(s)=\frac{1}{\mu_i}\ln(1+\mu_i s) \quad i=1,2.
    $$
Then we have
    \begin{eqnarray}
    u &=& g_1(w_1)=g_2(w_2),                \label{a1}
    \\
    e^{\mu_i u}&=&1+\mu_i w_i, \quad i=1,2. \label{a2}
    \end{eqnarray}
%  \vspace{1mm}  
Direct calculations give us
    \begin{eqnarray*}
    -\Delta w_i &=&  e^{\mu_i u}(\lambda c(x)u +h(x)+ac(x))+ e^{\mu_i u}(\mu(x)-\mu_i)|\nabla u|^2 
    \\
        &=& (1+\mu_iw_i)(\lambda c(x)g_i(w_i) +h(x)+ac(x))+ (1+\mu_iw_i)(\mu(x)-\mu_i)|\nabla u|^2.
    \end{eqnarray*}
Since $\mu_1\leq \mu(x)\leq \mu_2$, we have 
    \begin{eqnarray}
    -\Delta w_1 &\geq& (1+\mu_1w_1)[\lambda c(x)g_1(w_1) +h(x)+a c(x)],    \label{a3}
    \\
    -\Delta w_2 &\leq& (1+\mu_2w_2)[\lambda c(x)g_2(w_2) +h(x)+a c(x)], \label{a4}
    \end{eqnarray}
    in a weak sense.
    
From the inequalities \eqref{a3} and \eqref{a4}, we shall deduce that $w_2$ is uniformly bounded in $H^1_0(\Omega)$.
This will lead to the proof of the theorem by Lemma \ref{Sirakov}. We shall denote by  $C$ a generic constant independent of $\Lambda_1$ and by $C(\Lambda_1)$, a generic constant depending on $\Lambda_1$. We then precise its dependence on $\Lambda_1$.
\medbreak
 
We divide the proof into three steps.
\medskip

\noindent
{\bf Step 1.}
{\it Let $\theta=(\mu_2-\mu_1)/ \mu_2\in \,]0,1[$.  Then there exists $D=D(\Lambda_1)>0$ independent of $\lambda \in [\Lambda_1, \Lambda_2]$ and of
$a\in [0,A]$ such that}
    \begin{eqnarray}
    &\displaystyle \int_\Omega (1+\mu_1w_1^+)[cg_1(w_1^+) +h^++a c]\psi_1\, dx \leq D(\Lambda_1),            \label{a5}
		\\
    &\displaystyle  \int_\Omega (1+\mu_2w_2^+)^{1-\theta}[cg_2(w_2^+) +h^++a c]\psi_1 \, dx\leq D(\Lambda_1). \label{a6}
    \end{eqnarray}
{\it Moreover  $D(\Lambda_1) = O_{0^+}(e^{1/ \Lambda_1})$.} \medskip

\noindent Indeed, using $\psi_1 >0$ (defined in \eqref{eigenvaluep2}) as a test function in \eqref{a3} and integrating  we have 
$$
  \int_\Omega [\tilde\nu_1 c - \mu_1 h^-]w_1\psi_1 \, dx
 \geq    \int_\Omega (1+\mu_1w_1)[\lambda cg_1(w_1)+h +a c]\psi_1 \, dx.
$$
Recording that $\lambda \leq \Lambda_2$ and then, by Lemma \ref{lem 4.1}, that $g_1(w_1^-) = u^-$ is uniformly bounded  we then obtain
\arraycolsep1.5pt
$$
\begin{array}{rcl}
 \displaystyle
   \tilde\nu_1\int_\Omega cw_1\psi_1  \, dx
		& \geq &  
 \displaystyle
		\int_\Omega (1+\mu_1w_1)[\lambda cg_1(w_1)+h+a c]\psi_1 \, dx
%\\  
%&&\displaystyle
%\hfill
		+\mu_1 \int_{\Omega} h^- w_1\psi_1\, dx
		\\[3mm]
		&=&
 \displaystyle
		\int_\Omega (1+\mu_1w_1)[\lambda cg_1(w_1)+h^+ +a c]\psi_1 \, dx
		- \int_{\Omega} h^- \psi_1\, dx
\\  
&&\displaystyle
\hfill
		+\mu_1 \int_{\Omega} h^- w_1\psi_1\, dx
		\\[3mm]
		& \geq & 
				\displaystyle
				\int_\Omega (1+\mu_1w_1^+)[\lambda cg_1(w_1^+)+h^+ +a c]\psi_1\, dx - C.
   \end{array}
$$
\arraycolsep5pt
Since $\lambda \geq \Lambda_1$ we then deduce that
\begin{equation}
\label{eqa10}    
   \tilde\nu_1\int_\Omega cw_1\psi_1  \, dx
 \geq 
		\Lambda_1 \int_\Omega (1+\mu_1w_1^+)[cg_1(w_1^+)+h^+ +a c]\psi_1  \, dx-C.
   \end{equation}
	
Note that for any $\varepsilon >0$ there exists $C_{\varepsilon} >0$ such that, for all $t>0$,
\begin{equation}
\label{esti} 
t \leq \varepsilon (1 + \mu_1 t) g_1(t) + C_{\varepsilon}.
\end{equation}
A direct calculation shows that we can assume that $C_{\varepsilon}= O_{0^+}(\varepsilon e^{1/ \varepsilon})$.
Using  \eqref{esti} with $ \displaystyle  \varepsilon = \frac{\Lambda_1}{2 \tilde\nu_1}$, we get that 
\arraycolsep1.5pt
\begin{eqnarray}
\label{eqa11}    
    \tilde\nu_1\int_\Omega cw_1\psi_1 \, dx& \leq & \tilde\nu_1 \int_\Omega c w_1^+ \psi_1 \, dx\nonumber \\
        & \leq &  \frac{\Lambda_1}{2} \int_\Omega (1+\mu_1w_1^+)[cg_1(w_1^+)+h^+ +a c]\psi_1  \, dx+ C_{\Lambda_1}.
\end{eqnarray}
\arraycolsep5pt
We then obtain \eqref{a5} from \eqref{eqa10} and \eqref{eqa11}. Now observe that by  \eqref{a2}, 
    $$  
		1+\mu_1w_1= e^{\mu_1u}=(e^{\mu_2 u})^{1-\theta} = (1+\mu_2w_2)^{1-\theta}.
    $$
Thus from \eqref{a1} we see that \eqref{a6} is nothing but \eqref{a5}.
\medskip

\noindent
{\bf Step 2.}
{\it There exists a constant $D=D(\Lambda_1)>0$ independent of $a\in [0,A]$  and $\lambda \in [\Lambda_1, \Lambda_2]$ such that}
    \begin{equation}\label{unifestimates}
    \|\nabla w_2^+\|_2 \leq D(\Lambda_1).
   \end{equation}
   {\it Moreover $D(\Lambda_1) = O_{0^+}(e^{\beta/ \Lambda_1})$ with $\beta=\frac{\alpha}{2-q(1-\alpha)}$. } \medskip

First we use Lemma \ref{tec1} to choose $r \in\, ]0,1[$ and  $\alpha \in \, ]0, \frac{p-1}{2p-1}[$ such that $q$ and $\tau$ defined by \eqref{a0}
satisfy \eqref{boundsq} and \eqref{boundsalpha}.

Using $w_2^+$ as a test function in \eqref{a4} it follows that
    $$
    \|\nabla w_2^+\|_2^2
    \leq \int_\Omega(1+\mu_2w_2^+)[\lambda cg_2(w_2^+)+h^++ a c]w_2^+\, dx.
	$$
Setting  $H = h^+ + Ac$, we have 
$$
    \|\nabla w_2^+\|_2^2
    \leq  \Lambda_2 \int_\Omega(1+\mu_2w_2^+)[cg_2(w_2^+)+ H]w_2^+\, dx.
	$$
Now using H\"older's inequality, and since $ w_2^+\leq (1+\mu_2w_2^+)/\mu_2^{-1}$ 
we obtain using \eqref{a6} of Step 1 and for a $ D(\Lambda_1)= O_{0^+}(e^{1/ \Lambda_1})$,
\arraycolsep1.5pt
\begin{eqnarray*}
    \|\nabla w_2^+\|_2^2 & \leq & \frac{\Lambda_2}{\mu_2}\! \int_\Omega (1+\mu_2w_2^+)[c g_2(w_2^+)+ H)]\frac{\psi_1^\alpha}{(1+\mu_2w_2^+)^{\theta\alpha}\!\!}
                \frac{(1+\mu_2w_2^+)^{1+\theta\alpha}\!}{\psi_1^\alpha} dx
            \\
            &\leq& \frac{\Lambda_2}{\mu_2}\left(\int_\Omega (1+\mu_2w_2^+)[cg_2(w_2^+)+ H]\frac{\psi_1}{(1+\mu_2w_2^+)^\theta}\,dx\right)^\alpha 
    \\
    && \qquad \times
        \left(\int_\Omega (1+\mu_2w_2^+)[cg_2(w_2^+)+  H]\frac{(1+\mu_2w_2^+)^{\frac{1+\theta\alpha}{1-\alpha}}}{\psi_1^{\frac{\alpha}{1-\alpha}}}\,dx\right)^{1-\alpha} 
        \\
    &\leq& \frac{\Lambda_2}{\mu_2} D(\Lambda_1)^\alpha \!
    \left(\int_\Omega (1+\mu_2w_2^+)[cg_2(w_2^+)+  H]\frac{(1+\mu_2w_2^+)^{\frac{1+\theta\alpha}{1-\alpha}}}{\psi_1^{\frac{\alpha}{1-\alpha}}}\,dx\right)^{\!\!1-\alpha}\!\!\!.
    \end{eqnarray*}
\arraycolsep5pt 
%Here we used \eqref{a6}.
We note that for  $r>0$ given by Lemma~\ref{tec1}, there exists %some 
$C>0$
    $$  g_2(t) \leq t^r +C \quad \mbox{for all}\ t\geq 0.
    $$
Thus, direct calculations shows that
$$	(1+\mu_2w_2^+)[cg_2(w_2^+)+ H](1+\mu_2w_2^+)^{\frac{1+\theta\alpha}{1-\alpha}} 
		\leq (c+ H)({w_2^+}^q+C),
$$
where $q$ %and $\tau$ are 
is given in \eqref{a0}.
Therefore for some $D(\Lambda_1)=  O_{0^+}(e^{1/ \Lambda_1})$,
    $$   \|\nabla w_2^+\|_2^2
		\leq D(\Lambda_1)^{\alpha}\left[\left(\int_\Omega \left(\frac{(c+ H)^{1/q}w_2^+}{\psi_1^\tau}\right)^q\,dx\right)^{1-\alpha}+1\right],
    $$
 with $q$ and $\tau$ given in \eqref{a0}. Applying Lemma \ref{tec2}, we then obtain    
    $$   
		\|\nabla w_2^+\|_2^2 \leq D(\Lambda_1)^{\alpha}\left[\|c + H \|_p^{q(1-\alpha)}\|\nabla w_2^+\|_2^{q(1-\alpha)} +1\right].
		$$
    By \eqref{boundsalpha}, we have $q(1-\alpha)<2$ and this concludes the proof of Step 2.
\medskip

%\noindent
%{\bf Step 4.}
%{\it There exists a constant $C>0$ independent of $a\in [0,A]$  and $\lambda \in [\Lambda_1, \Lambda_2]$ such that}
%    \begin{equation}\label{unifestimatesbas}
%    \|\nabla w_1^-\|_2 \leq C.
%    \end{equation}
%Using $w_1^-$ as test function in \eqref{a3}, we have
%$$
%\begin{array}{rcl}
%\|\nabla w_1^-\|_2^2 
%&\leq&
%\displaystyle
%- \int_{\Omega} (1+\mu_1 w_1^-)(\lambda c(x) g_1(w_1^-)+h(x)+a c(x)) w_1^-
%\\
%&\leq&
%\displaystyle
%- \int_{\Omega} (1+\mu_1 w_1^-)(\lambda c(x) g_1(w_1^-)+h(x)+a c(x)) w_1^-
%\
%&\leq& 
%\displaystyle
%C \|c\|_{N/2}\|\nabla w_1^-\| + C \|h\|_{N/2} \|\nabla w_1^-\|
%\end{array}
%$$
%and thus Step 4 hold.
%\medbreak

%\noindent
%{\bf Step 5.}
%{\it There exists a constant $C>0$ independent of $a\in [0,A]$ and of $\lambda \in [\Lambda_1, \Lambda_2]$  such that}
%    \begin{equation}\label{unifestimatesbas2}
 %   \|\nabla w_2^-\|_2 \leq C.
 %   \end{equation}
%By \eqref{a2}, we have
%$$
%w_i^-=-\frac{1}{\mu_i}(e^{-\mu_i u^-}-1)
%$$
%and hence
%$$
%|\nabla w_i^-\|_2=\|e^{-\mu_i u^-} \nabla u^-\|_2.
%$$
%As, by assumptions, $u\geq \alpha$, we obtain
%$$
%\|\nabla w_2^-\|_2
%\leq \|\nabla u^-\|_2
%\leq e^{\mu_1 |\min \alpha|}\|\nabla w_1^-\|_2
%$$
%and the result follows from Step 4.
%\medbreak

\noindent 
{\bf Step 3.} {\it Conclusion.}
\medskip

By Lemma \ref{lem 4.1} we already know that $u > - M$ for some $M>0$. Hence
we just have to show that the estimate \eqref{unifestimates} derived in Step 2
gives an   estimate in the  $L^{\infty}(\Omega)$ norm of $w_2^+$. Since $w_2$ satisfies \eqref{a4} we can use Lemma \ref{Sirakov}
with
$$d=(1+\mu_2 w_2)\lambda c\frac{\ln(1+\mu_2 w_2)}{\mu_2 w_2}+\mu_2\,(h+A\, c) 
$$
and 
$$
f=h+A\, c.
$$
Observe that, for any $r\in \,]0,1[$, there
exists $C >0$ such that, for all $x\in \Omega$ and all $\lambda\leq \Lambda_2$,
$$
\lambda c \Big | (1+\mu_2 w_2) \frac{\ln(1+\mu_2 w_2)}{\mu_2 w_2}\Big | \leq C c (|w_2|^r+1),
$$
where $C$ depends on $\Lambda_2$, $r$, $\mu_2$.
\medbreak

Thus, since $c(x) \in L^p(\Omega)$ with $p > \frac{N}{2}$ and  $w_2$ is bounded in $L^{\frac{2N}{N-2}}(\Omega)$,  taking
$r >0$ sufficiently small  we see, using  H\"older's inequality, that  $c(x) |w_2(x)|^r \in L^{p_1}(\Omega)$ for some $p_1 > \frac{N}{2}$. Now as 
$h\in L^p(\Omega)$ for some $p>\frac{N}{2}$, clearly all the assumptions of Lemma \ref{Sirakov} are satisfied.  From \eqref{unifestimates}  we then deduce that there exists a constant $D(\Lambda_1) >0$ with $D(\Lambda_1)= O_{0^+}(e^{\beta/ \Lambda_1}) $ and $\beta$ given by Step 2, such that
$$
\|w_2^+\|_{\infty} \leq D(\Lambda_1).
$$
Now since $u^+ = g_2(w_2^+)$ we deduce that
$$
\|u^+\|_{\infty} \leq M(\Lambda_1)
$$
for some $M(\Lambda_1) = O_{0^+}(1 / \Lambda_1)$. 
\end{proof}

\begin{lem}
\label{non ex}
For every $\Lambda_2>0$, there exists $A_1>0$, independent of $\lambda \in [0, \Lambda_2]$, such that the problem \eqref{EE1a} has no solution  for $a\geq A_1$.
\end{lem}

\begin{proof}
Let $\phi\in C_0^{\infty}(\Omega)$ such that $\int_{\Omega} c(x)\phi^2 \, dx>0$ and use $\phi^2$ as test function in \eqref{EE1a}. Then we obtain
$$
\begin{array}{rcl}
\displaystyle
\int_{\Omega} \frac{1}{|\mu(x)|} |\nabla \phi|^2\,dx
&\geq&
\displaystyle
2\int_{\Omega} \phi \nabla u \nabla \phi \,dx - \int_{\Omega} |\mu(x)| |\nabla u|^2 \phi^2 \,dx
\\
&=&
\displaystyle
\lambda\int_{\Omega} c \,u\, \phi^2 \,dx+\int_{\Omega} h\,\phi^2\,dx+a\int_{\Omega} c\,\phi^2 \,dx
\\
&\geq&
\displaystyle
\lambda\min u \int_{\Omega} c\,  \phi^2 \,dx+\int_{\Omega} h\,\phi^2 \,dx+a\int_{\Omega} c\,\phi^2\,dx.
\end{array}
$$
Since, by Lemma \ref{lem 4.1}, there exists $M>0$ such that, for all  $a\geq 0$,  any solution $u$ satisfies $u >-M$, this gives a contradiction for $a>0$ large enough.
\end{proof}

\section{Results}
\label{Section3}

This section is devoted to the proof of our main results.

\begin{proof}[Proof of Theorem \ref{thm 0}]
Let ${\mathcal C}^+ \subset \Sigma$ be the continuum obtained in Theorem \ref{ADJT1}. Either its projection  $\mbox{\rm Proj}_{\R} {\mathcal C}^+$ 
on the $\lambda$-axis is $\mathbb R$ or its projection  
on the $\lambda$-axis is $]-\infty,\overline \lambda]$ with $0<\overline\lambda<+\infty$. In the first case, the result is proved. In the second case, as by  Theorem \ref{ADJT1} we know that 
${\mathcal C}^+$ is unbounded, its projection  on $C(\overline\Omega)$ has to be unbounded. 
\medbreak

By Theorem \ref{bounds1} we know that for every $0<\Lambda_1 <\Lambda_2$,  there is  an a priori bound on the solutions for 
$\lambda\in [\Lambda_1,\Lambda_2]$.  This means that the projection of ${\mathcal C}^+\cap ( [\Lambda_1, \Lambda_2] \times C(\overline\Omega))$ on $C(\overline\Omega)$ is bounded.  Now by Lemma \ref{lem 4.1} there is a lower bound on the solutions for $\lambda\leq \Lambda_2$. Thus ${\mathcal C}^+$ must emanate from infinity to the right of $\lambda=0$ with the positive part of the corresponding solution blowing up to infinity.
\end{proof}

\begin{cor}
\label{Cor 4.1}
Under assumption $(A)$  and assuming that $(P_0)$ has a solution,  let $\varphi_1 >0$ the first eigenfunction of \eqref{eigenvaluep}. If 
 $$
\int_{\Omega} h \varphi_1 \, dx \geq 0 ,
$$
then  we are in case (i) of Theorem \ref{thm 0} and $\max \mbox{\rm Proj}_{\R} {\mathcal C}^+ <\gamma_1$.
\end{cor}

\begin{proof}
Let $u$ be a solution of $(P_{\lambda})$. Multiplying by $\varphi_1 >0$ and integrating we have
$$
(\gamma_1-\lambda) \int_{\Omega} cu\varphi_1 \, dx = \int_{\Omega} \mu |\nabla u|^2 \varphi_1 \, dx + \int_{\Omega} h \varphi_1 \, dx >0
$$
which is a contradiction for $\lambda=\gamma_1$. Hence $(P_{\lambda})$ has no solution for $\lambda=\gamma_1$ which proves that we are in the first situation in Theorem \ref{thm 0}.
\end{proof}

In order to consider the situation where $(P_0)$ has a  solution with $\min u<0$, we need the following lemmas.

\begin{lem}
\label{lem lower}
Under assumption $(A)$, for every $\lambda \geq 0$, there exists a strict lower solution $v_{\lambda}$ of $(P_{\lambda})$ such that,  
every upper solution $\beta$ of $(P_{\lambda})$ satisfies $v_{\lambda} \leq\beta$.
\end{lem} 

%\begin{remark}
%The idea comes from \cite{AbPePr}.
%\end{remark}

\begin{proof}
Let $M>0$ be given by Corollary \ref{cor 4.1b} such that, for every upper solution $\beta$ of 
\begin{equation}
\label{eq **}
\begin{array}{cl}
-\Delta u = \lambda c(x) u+ \mu(x) |\nabla u|^2 -h^-(x) -1,&\mbox{in }\Omega,
\\
u=0,&\mbox{on }\partial\Omega,
\end{array}
\end{equation}
we have $\beta\geq -M$.

Let $k>M$ and consider $\alpha_k$ the solution of
$$
\begin{array}{cl}
-\Delta v = - \lambda k c(x) -h^-(x) -1,&\mbox{in }\Omega,
\\
v=0,&\mbox{on }\partial\Omega.
\end{array}
$$
As $- \lambda k c(x) -h^-(x) -1< 0$, we have $\alpha_k\ll0$ by the strong maximum principle. 
\medbreak

\noindent{\it Claim 1: Every  upper solution $\beta$ of $(P_{\lambda})$ satisfies $\beta\geq \alpha_k$.} In fact  $\beta=\min\{\beta_j\mid 1\leq j \leq l\}$ where $\beta_1, \ldots,  \beta_l$ are regular upper solutions of $(P_{\lambda})$. Setting $w=\beta_j-\alpha_k$ for some $1\leq j \leq l$ we have
$$
\begin{array}{cl}
-\Delta w \geq \lambda c(x) (\beta_j+k)\geq0,&\mbox{in }\Omega,
\\
w=0,&\mbox{on }\partial\Omega.
\end{array}
$$
By the maximum principle $w\geq0$ i.e. $\beta_j\geq\alpha_k$. This proves the Claim.
\medbreak

Consider then the problem 
\begin{equation}
\label{eq *}
\begin{array}{cl}
-\Delta v = \lambda c(x) T_k(v) + \mu(x)|\nabla v|^2 - h^-(x)-1,&\mbox{in }\Omega,
\\
v=0,&\mbox{on }\partial\Omega,
\end{array}
\end{equation}
where 
$$
\begin{array}{rcll}
T_k(v)&=&-k,&\mbox{if }v\leq -k,
\\
&=&v,&\mbox{if }v> -k.
\end{array}
$$
It is easy to prove that $\alpha_k$ and $\beta$ are lower and upper solutions of \eqref{eq *} and hence, by Theorem \ref{sousDe}, 
this problem has a minimal solution $v_k$ with $\alpha_k\leq v_k\leq\beta$.
\medbreak

{\it Claim 2: Every  upper solution $\beta$ of $(P_{\lambda})$ satisfies $\beta\geq v_k$.}
Observe that, by construction of \eqref{eq *},  every upper solution $\beta$ of  $(P_{\lambda})$ is also an upper solution of \eqref{eq *}. 
As, by Claim 1, we have $\beta\geq \alpha_k$, the minimality of $v_k$ implies that $v_k\leq\beta$.
\medbreak

{\it Claim 3: $v_k$ is a lower solution of $(P_{\lambda})$.}
Observe that $v_k$ is an upper solution of \eqref{eq **}. Hence $v_k\geq -M>-k$ and  $v_k$ satisfies
$$
\begin{array}{cl}
-\Delta v = \lambda c(x) v + \mu(x)|\nabla v|^2 - h^-(x)-1,&\mbox{in }\Omega,
\\
v=0,&\mbox{on }\partial\Omega.
\end{array}
$$
This implies that $v_k$ is a lower solution of $(P_{\lambda})$.
\medbreak

{\it Claim 4: 
$v_k$ is a strict lower solution of $(P_{\lambda})$.} Let $u$ be a solution of $(P_{\lambda})$ with $u\geq v_k$. Then $w=u-v_k$ satisfies
$$
\begin{array}{cl}
-\Delta w - \mu(x)\langle \nabla u + \nabla v_k\mid \nabla w\rangle\geq  \lambda c(x) w + h^+(x)+1 \geq 1,&\mbox{in }\Omega,
\\
w=0,&\mbox{on }\partial\Omega.
\end{array}
$$
By the strong maximum principle (see \cite[Theorem 3.27]{Tr87}), we deduce that $w\gg0$ i.e. $u\gg v_k$.
\end{proof}

\begin{remark}
Lemma \ref{lem lower} shows that, for $(P_0)$, to have an upper solution is equivalent to have a solution.
\end{remark}

\begin{proof}[Proof of Theorem \ref{thmlocal}] 

We proceed in several steps. 
\medskip

\noindent{\it Step 1:  For all $\varepsilon>0$, there exists $R>0$ such that $\deg(I-{\mathcal M}_{0}, {\mathcal S}) = 1$ 
with}
$$
{\mathcal S} =
\{u \in C^1_0(\overline\Omega) \mid u_0-\varepsilon \ll u \ll u_0+\varepsilon, \,\, \|u\|_{C^1}<R \}.
$$
It is easy to prove that $u_0-\varepsilon$ and $u_0+\varepsilon$ are lower and upper solutions of $(P_{0})$. Moreover, as $u_0$ is the unique solution of  $(P_{0})$, we deduce that $u_0-\varepsilon$ and $u_0+\varepsilon$ are strict lower and upper solutions of $(P_{0})$.
The result then  follows by Theorem \ref{sousDe}.
\medskip

\noindent{\it Step 2: There exists a $\lambda_0 >0$ such that 
 $\deg(I-{\mathcal M}_{\lambda}, {\mathcal S}) = 1$ 
 for $\lambda \in \,]0, \lambda_0[$ with ${\mathcal S}$ defined in Step 1.} \medskip

Let us prove first the existence of $\lambda_0 >0$ such that, for $\lambda \in\, ]0, \lambda_0[$, $(P_{\lambda})$ has no solution on $\partial{\mathcal S}$.
Otherwise, there exist a sequence $\{\lambda_n\}$ with $\lambda_n\to 0$ and a corresponding sequence of solution $\{u_n\}\subset W^{2,p}(\Omega)$ of $(P_{\lambda})$ with  
$u_n\in \partial{\mathcal S}$.
Increasing  $R$ if necessary, this means that $u_0-\varepsilon \leq u_n\leq  u_0+\varepsilon$ and either $\max (u_n-u_0)=\varepsilon$ or $\min (u_n-u_0)=-\varepsilon$.
By Lemma \ref{lem 2} there exists a $R>0$ such that, for all $n \in \mathbb N$,  $\|u_n\|_{W^{2,p}}<R$. 
Hence, up to a subsequence, $u_n\to u$ in $C^1_0(\overline\Omega)$. 
From this strong convergence we easily observe that
$$
\begin{array}{cl}
-\Delta u = \mu(x)|\nabla u|^2+h(x), &\mbox{ in }\Omega,
\\
 u=0,  &\mbox{ on }\partial \Omega,
\end{array}
$$
and either  $\max (u-u_0)=\varepsilon$ or $\min (u-u_0)=-\varepsilon$ i.e. $u$ is a solution of $(P_0)$ with $u\in \partial{\mathcal S}$ which contradicts Step 1.
\smallbreak

We conclude by the invariance by homotopy of the degree that 
$$
\deg(I-{\mathcal M}_{\lambda}, {\mathcal S})= \deg(I-{\mathcal M}_{0}, {\mathcal S})=1.
$$
%\medbreak

\noindent{\it Step 3: $(P_{\lambda})$ has two solutions when $\lambda \in\, ]0, \lambda_0[$.}
By Step 2, the existence of a first solution $u_0-\varepsilon \ll u_{\lambda,1}\ll  u_0+\varepsilon$ is proved.

Also, using Lemma \ref{non ex}, there exists $A_1 >0$ large enough such that \eqref{EE1a} has no solution for $a\geq A_1$.
By Theorem \ref{bounds1} and Lemma \ref{lem 2} there exists a  $R_0 >R >0$ such that, for all $a\in [0,A_1]$, 
every solution of \eqref{EE1a}  satisfies  $\|u\|_{C^1}<R_0$. Hence, by homotopy invariance of the degree, we have
$$
\deg(I-{\mathcal M}_{\lambda}, B(0,R_0)) = \deg(I-{\mathcal M}_{\lambda}-{\mathcal L}^{-1}(A_1c), B(0,R_0)).
$$
As for $a=A_1$, the problem \eqref{EE1a} has no solution, we have
$$
\deg(I-{\mathcal M}_{\lambda}-{\mathcal L}^{-1}(A_1c),B(0,R_0))=0.
$$
We then conclude that
$$
\deg(I-{\mathcal M}_{\lambda},B(0,R_0) \setminus {\mathcal S})=\deg(I-{\mathcal M}_{\lambda},B(0,R_0))-\deg(I-{\mathcal M}_{\lambda},{\mathcal S})=-1.
$$
This proves the existence of a second solution $u_{\lambda,2}$ of $(P_{\lambda})$ with $u_{\lambda,2}\in  B(0,R_0) \setminus {\mathcal S}$.
%By Proposition \ref{unique}, we have $\max u_{\lambda,2}>0$.
%
%Having obtained in Step 1 a strict upper solution this result is derived exactly as in Step 3 of the proof of Theorem \ref{thm 2}.
\medbreak

\noindent{\it Step 4: Existence of $\overline\lambda$ such that, for all $\lambda\in \,]0,\overline\lambda[$, the problem $(P_{\lambda})$   has at least two solutions 
with $u_{\lambda,1}\ll u_{\lambda,2}$.}
Define 
$$
\overline\lambda=\sup\{\mu \mid \forall \lambda\in \,]0,\mu[, \,\, (P_{\lambda}) \mbox{ has at least two solutions}\}.
$$ 
For $\lambda\in \,]0,\overline\lambda[$, $(P_{\lambda})$ has at least two solutions $u_{\lambda,1}$ and $u_{\lambda,2}$. 
Let us consider the strict lower solution $\alpha$ given by Lemma \ref{lem lower}.
As $\alpha\leq u$ for all $u$ solution of $(P_{\lambda})$, we can choose $u_{\lambda,1}$ as the minimal solution with $u_{\lambda,1}\geq \alpha$. Hence we have 
$u_{\lambda,1}\lneqq u_{\lambda,2}$ as otherwise  there exists a solution $u$ with $\alpha\leq u \leq \min(u_{\lambda,1},u_{\lambda,2})$ which contradicts the minimality of 
$u_{\lambda,1}$.
\smallbreak

Now observe that, by convexity of $y\mapsto |y|^2$, the function $\beta=\frac12(u_{\lambda,1}+u_{\lambda,2})$ is an upper solution of $(P_{\lambda})$ which is not a solution. 
Let us prove that $\beta$ is a strict upper solution of  $(P_{\lambda})$. Let $u$ be a solution of $(P_{\lambda})$ with $u\leq \beta$. 
Then  $v:=\beta-u$ satisfies
$$
\begin{array}{cl}
-\Delta v - \mu(x) \langle \nabla \beta+\nabla u\mid \nabla v\rangle\geq \lambda c(x) v\geq 0,&\mbox{in }\Omega,
\\
 v\geq 0,&\mbox{in }\Omega.
\end{array}
$$
By the strong maximum principle  we deduce that either $v\gg0$ or $v\equiv0$. If $v\equiv0$, then $\beta=u$ is solution which contradicts the construction of $\beta$.
As $u_{\lambda,1}\lneqq \beta\lneqq u_{\lambda,2}$, we deduce from the fact that $\beta$ is strict that $u_{\lambda,1}\ll \beta\lneqq u_{\lambda,2}$ and hence we have proved the step. 
\medbreak

\noindent{\it Step 5: In case $\overline\lambda<+\infty$, the problem $(P_{\overline\lambda})$ has at least one solution  $u$.}
Let $\{\lambda_n\} \subset \, ]0, \overline\lambda[$ be a sequence such that $ \lambda_n \to  \overline\lambda$ 
and $\{u_n \} \subset W^{2,p}(\Omega)$ be a sequence of corresponding  solutions. 
By Theorem \ref{bounds1}, there exists a constant $M>0$ such that, for all $n \in \mathbb N$, $\|u_n\|_{\infty}<M$ and, by Lemma \ref{lem 2}, 
we have $R>0$ such that, for all $n \in \mathbb N$,  $\|u_n\|_{W^{2,p}}<R$. 
%As $u_n$ satisfies
%$$
%\begin{array}{cl}
%-\Delta u_n = \lambda_n c(x)u_n+\mu(x)|\nabla u_n|^2+h(x), &\mbox{ in }\Omega,
%\\
% u_n=0,  &\mbox{ on }\partial \Omega,
%\end{array}
%$$
%we deduce that $\{\|u_n\|_{W^{2,p}}\}$ is bounded. 
Hence, up to a subsequence, $u_n\to u$ in $C^1_0(\overline\Omega)$. 
From this strong convergence we easily observe that
$$
\begin{array}{cl}
-\Delta u = \overline\lambda c(x)u+\mu(x)|\nabla u|^2+h(x), &\mbox{ in }\Omega,
\\
 u=0,  &\mbox{ on }\partial \Omega,
\end{array}
$$
namely $u \in W^{2,p}(\Omega)$ is a solution of $(P_{\overline\lambda})$. 
\medbreak

\noindent{\it Step 6: Uniqueness of the solution of $(P_{\overline\lambda})$ in case $\overline\lambda<+\infty$.}
Otherwise, if we have two distincts solutions $u_1$ and $u_2$ of $(P_{\overline\lambda})$,    then, as in Step 4, we prove that  $\beta=\frac12(u_1+u_2)$ is a strict  upper solution of $(P_{\overline\lambda})$. Let us consider the strict lower solution $\alpha\ll\beta$ of $(P_{\overline\lambda})$ given by Lemma \ref{lem lower}.
By Theorem \ref{sousDe}, we then have $R>0$ such that
$$
\deg(I-{\mathcal M}_{\overline\lambda}, \tilde{\mathcal S}) = 1,
$$
where
$$
\tilde {\mathcal S} =
\{u \in C^1_0(\overline\Omega) \mid \alpha \ll u \ll \beta, \quad \|u\|_{C^1}<R \}.
$$
Arguing as in Step 2, we prove the existence of $\varepsilon>0$ such that, for all $\lambda\in [\overline\lambda-\varepsilon, \overline\lambda+\varepsilon]$,
 $\deg(I-{\mathcal M}_{\lambda}, \tilde{\mathcal S}) = 1$ and,  as in Step 3, we prove that $(P_{\lambda})$  has at least two solutions for 
$\lambda\in [\overline\lambda, \overline\lambda+\varepsilon]$
which contradicts the definition of $\overline\lambda$.
\medbreak

\noindent{\it Step 7: Behaviour of the solutions for $\lambda\to0$.}
Let $\{\lambda_n\}\subset \,]0,\overline\lambda[$ be a decreasing sequence such that $\lambda_n\to 0$. Without loss of generality, we suppose $\lambda_n\in \,]0,\lambda_0[$.
Then, by Steps 2 and 4, the corresponding solutions $u_{\lambda_n,1}$ satisfy $u_{\lambda_n,1}\leq u_0+\varepsilon$. Recall that, by Corollary \ref{cor 4.1b}, there exists $M>0$ such that, for all $n$, $u_{\lambda_n,1}\geq -M$.
This implies that  the sequence $\{u_{\lambda_n,1}\}$  is bounded in $C(\overline\Omega)$.  We argue then as  in Step 5 to prove that
$u_{\lambda_n,1}\to u$ in $C^1_0(\overline\Omega)$ with $u$ solution of $(P_0)$. By uniqueness of the solution of $(P_0)$, we deduce that $u=u_0$.
\smallbreak

Now let us consider the sequence  $\{u_{\lambda_n,2}\}$. If $\{u_{\lambda_n,2}\}$ is bounded, then as in Step 5, we have that 
$u_{\lambda_n,2}\to u$ in $C^1_0(\overline\Omega)$ with $u$ solution of $(P_0)$. By Step 3 and the facts that $u_{\lambda_n,2}\not\in {\mathcal S}$, 
$u_{\lambda_n,2}\gg u_{\lambda_n,1}$ and $u_{\lambda_n,1}\to u_0$, we know that $\max\{u_{\lambda_n,2}-u_0\}>\varepsilon$. This implies that $u\not=u_0$ which contradicts the uniqueness of the solution of $(P_0)$.
\end{proof}

\begin{remark}
Observe that, by the above proof, we see that the set of $\lambda$ for which the problem $(P_{\lambda})$  has at least two solutions is open in $]0,+\infty[$.
\end{remark}

\begin{proof}[Proof of Theorem \ref{thm 1}]

We proceed in several steps. \medskip

\noindent{\it Step 1: Every non-negative upper solution of $(P_{\lambda})$ satisfies $u\gg u_0$.}
If $u$ is a non-negative upper solution of $(P_{\lambda})$ then $u$ is an upper solution of $(P_0)$. By Proposition \ref{comparison} we deduce that  $u\geq u_0$ and hence $u$ is not a solution of  $(P_0)$. As in Step 4 of the proof of Theorem \ref{thmlocal}, we prove that $u\gg u_0$.
\medbreak

\noindent{\it Step 2: The problem $(P_{\lambda})$ has no non-negative solution for $\lambda$ large.}
Let $\varphi_1 >0$ the first eigenfunction of \eqref{eigenvaluep}.
If $(P_{\lambda})$ has a non-negative solution, multiplying $(P_{\lambda})$ by $\varphi_1 >0$ and integrating we obtain
$$
\gamma_1 \int_{\Omega} cu\varphi_1 \, dx =-\int_{\Omega}\Delta u \varphi_1 \, dx = \lambda \int_{\Omega} cu\varphi_1 \, dx +\int_{\Omega} \mu |\nabla u|^2 \varphi_1 \, dx + \int_{\Omega} h\varphi_1 \, dx,
$$
and hence, for $\lambda>\gamma_1$, as $u\geq u_0$, we have
$$
\begin{array}{rcl}
0&\geq&
\displaystyle
(\lambda-\gamma_1) \int_{\Omega} cu\varphi_1 \, dx +\int_{\Omega} \mu |\nabla u|^2 \varphi_1 \, dx + \int_{\Omega} h\varphi_1 \, dx
\\[3mm]
&\geq&
\displaystyle
(\lambda-\gamma_1) \int_{\Omega} cu_0\varphi_1 \, dx +\int_{\Omega} \mu |\nabla u|^2 \varphi_1 \, dx + \int_{\Omega} h\varphi_1 \, dx
\end{array}
$$
which gives a contradiction for $\lambda$ large enough.
\medbreak

\noindent{\it Step 3: Define $\overline \lambda=\sup\{\lambda \mid (P_{\lambda}) \mbox{ has a solution } u_{\lambda}\geq0\}$, then, $\overline \lambda<+\infty$ and, 
for all $\lambda>\overline\lambda$, $(P_{\lambda})$ has no non-negative solution.}
This is obvious by definition of $\overline \lambda$ and Step 2.
\medbreak

\noindent{\it Step 4: For all $0<\lambda<\overline\lambda$, $(P_{\lambda})$ has well ordered strict lower and upper solutions.}
Observe that $u_0$ is a lower solution of $(P_{\lambda})$ which is not a solution.
By definition of $\overline\lambda$, we can find $\tilde \lambda \in \,]\lambda, \overline\lambda[$ and a non-negative solution $u_{\tilde\lambda}$ of $(P_{\tilde \lambda})$. Then  $u_{\tilde\lambda}$ is an upper solution of $(P_{\lambda})$ which is not a solution and satisfies $u_{\tilde\lambda}\geq u_0$ by Step 1.  At this point following the arguments of Step 4 of the proof of Theorem \ref{thmlocal}, we prove that $u_0$ and $u_{\tilde\lambda}$ are strict lower and upper solutions of $(P_{\lambda})$.
%By Theorem \ref{sousDe}, this proves already the existence of a first solution $u_{\lambda,1}$ of $(P_{\lambda})$  with $u_0\leq u_{\lambda,1}\leq u_{\tilde\lambda}$.
%
%
%Let us now prove that $u_0$ and $u_{\tilde\lambda}$ are respectively strict lower and upper solutions of  $(P_{\lambda})$. Let $u$ be a solution of $(P_{\lambda})$ 
%with $u\leq u_{\tilde\lambda}$. Then  $v:=u_{\tilde\lambda}-u$ satisfies
%$$
%\begin{array}{cl}
%-\Delta v - \mu(x) \langle \nabla u_{\tilde\lambda}+\nabla u\mid \nabla v\rangle\geq \lambda c(x) v\geq 0,&\mbox{in }\Omega,
%\\
% v\geq 0,&\mbox{in }\Omega.
%\end{array}
%$$
%By the strong maximum principle  we deduce that either $v\gg0$ or $v\equiv0$. If $v\equiv0$ then 
%$c(x)(\tilde\lambda -\lambda) u_{\tilde\lambda}\equiv 0$
%which contradicts $ u_{\tilde\lambda}\geq u_0$ and $cu_0 \gneqq 0$.
%
%In the same way we prove that $u_0$ is a strict lower solution.
\medbreak

\noindent{\it Step 5:  For all $\lambda \in \, ]0, \overline\lambda [$, $(P_{\lambda})$ has at least two positive solutions with $u_0\ll u_{\lambda,1}\ll u_{\lambda,2}$.}
By Step 4, Theorem \ref{sousDe} and Remark \ref{rem-utile}, we have $R_0>0$ such that $\deg(I-{\mathcal M}_{\lambda}, {\mathcal S}) = 1$ 
with
$$
{\mathcal S} =
\{u \in C^1_0(\overline\Omega) \mid u_0 \ll u \ll u_{\tilde\lambda}, \,\, \|u\|_{C^1}<R_0 \}.
$$
and we have  the existence of a first solution $u_{\lambda,1}$ of $(P_{\lambda})$  with $u_0\leq u_{\lambda,1}\leq u_{\tilde\lambda}$. Let us choose $u_{\lambda,1}$ as the minimal solution between  $u_0$ and $u_{\tilde\lambda}$.

Now, using Lemma \ref{non ex}, there exists $A_1>0$ large enough such that \eqref{EE1a} has no solution for $a\geq A_1$.
By Theorem \ref{bounds1} and Lemma \ref{lem 2} there exists  $R_1>R_0 >0$ such that, for  any $a\in [0,A_1]$, 
every solution of \eqref{EE1a} with $u\geq u_0$ satisfies $\|u\|_{C^1}<R_1$. Hence, by homotopy invariance of the degree we have
$$
\deg(I-{\mathcal M}_{\lambda}, {\mathcal D}) = \deg(I-{\mathcal M}_{\lambda}-{\mathcal L}^{-1}(A_1c), {\mathcal D}),
$$
where
$$
{\mathcal D}=\{u \in C^1_0(\overline\Omega) \mid u_0 \ll u, \,\, \|u\|_{C^1}<R_1 \}.
$$
As for $a=A_1$, \eqref{EE1a} has no solution, 
$
\deg(I-{\mathcal M}_{\lambda}-{\mathcal L}^{-1}(A_1c),{\mathcal D})=0$
and we obtain
$$
\deg(I-{\mathcal M}_{\lambda},{\mathcal D} \setminus {\mathcal S})=\deg(I-{\mathcal M}_{\lambda},{\mathcal D})-\deg(I-{\mathcal M}_{\lambda},{\mathcal S})= 0-1=-1.
$$
This proves the existence of a second solution $u_{\lambda,2}$ of $(P_{\lambda})$ with $u_{\lambda,2}\gg u_0$.
As $u_{\lambda,1}$ is the minimal solution  between $u_0$ and $u_{\tilde\lambda}$,  we have $u_{\lambda,1}\lneqq u_{\lambda,2}$ as otherwise, 
by Theorem \ref{sousDe}, we have a solution $u$ with $u_0\leq u \leq\min\{u_{\lambda,1}, u_{\lambda,2}, u_{\tilde\lambda}\}$ which contradicts the minimality of 
$u_{\lambda,1}$. We proceed as in  Step 4 of the proof of Theorem \ref{thmlocal} to conclude that $u_{\lambda,1}\ll u_{\lambda,2}$.
\medbreak

\noindent{\it Step 6: For $\lambda_1<\lambda_2$, we have $u_{\lambda_1,1}\ll u_{\lambda_2,1}$.}
As $u_{\lambda,1}$ is the minimal solution  above $u_0$ and,  as in Step 4, $u_{\lambda_2,1}$ is a strict upper solution of $(P_{\lambda_1})$  with 
$u_{\lambda_2,1}\geq u_0$, we deduce that
$u_{\lambda_1,1}\ll u_{\lambda_2,1}$.
\medbreak

\noindent{\it Step 7: The problem $(P_{\overline\lambda})$ has at least one solution.}
Let $\{\lambda_n\} \subset \, ]0, \overline\lambda[$ be a sequence such that $ \lambda_n \to  \overline\lambda$ 
and $\{u_n \} \subset W^{2,p}(\Omega)$ be a sequence of corresponding  non negative solutions. 
We argue as in Step 5 of the proof of Theorem \ref{thmlocal} to obtain that, up to a subsequence, $u_n\to u$ in $C^1_0(\overline\Omega)$ with $u \in W^{2,p}(\Omega)$ 
solution of $(P_{\overline\lambda})$.
\medbreak

\noindent{\it Step 8: Uniqueness of the non-negative solution of $(P_{\overline\lambda})$.}
The proof follows the lines of Step 6 of the proof of Theorem \ref{thmlocal}.
% Otherwise, if we have two distincts non negative solutions $u_1$ and $u_2$ of $(P_{\overline\lambda})$  then $\beta=\frac12(u_1+u_2)$ is an upper solution of 
%$(P_{\overline\lambda})$ by convexity of $y\mapsto |y|^2$. Moreover, as in Step 4, it is easy to see that $\beta$ is strict. By Theorem \ref{sousDe}, we then have $R>0$ such that
%$$
%\deg(I-{\mathcal M}_{\overline\lambda}, \tilde{\mathcal S}) = 1,
%$$
%where
%$$
%\tilde {\mathcal S} =
%\{u \in C^1_0(\overline\Omega) \mid u_0 \ll u \ll \beta, \quad \|u\|_{C^1}<R \}.
%$$
%Arguing as in Step 2 of the proof of
%Theorem \ref{thmlocal}, we prove the existence of $\varepsilon>0$ such that, for all $\lambda\in [\overline\lambda-\varepsilon, \overline\lambda+\varepsilon]$,
% $\deg(I-{\mathcal M}_{\lambda}, \tilde{\mathcal S}) = 1$.
%In particular $(P_{\lambda})$  has a solution $u\geq u_0$ for $\lambda\in [\overline\lambda, \overline\lambda+\varepsilon]$
%which contradicts the definition of $\overline\lambda$.
\medbreak

\noindent{\it Step 9: Behaviour of the solutions for $\lambda\to0$.}
This can be proved as in Step 7 of the proof of Theorem \ref{thmlocal}.
%Let $\{\lambda_n\}\subset \,]0,\overline\lambda[$ be a decreasing sequence such that $\lambda_n\to 0$. 
%The corresponding solutions satisfy $u_0\leq u_{\lambda_{n+1},1}\leq u_{\lambda_n,1}$. As this sequence is decreasing and bounded below, it is bounded in $C(\overline\Omega)$
%{\color{red}Hence we can argue as in Step 5 of the proof of Theorem \ref{thmlocal}} to prove that
%$u_{\lambda_n,1}\to u$ in $C^1_0(\overline\Omega)$ with $u$ solution of $(P_0)$. By uniqueness of the solution of $(P_0)$, we deduce that $u=u_0$.
%\medbreak
%
%Now, let us consider the sequence  $\{u_{\lambda_n,2}\}$. If $\{u_{\lambda_n,2}\}$ is bounded, {\color{red}then as in Step 5 of the proof of Theorem \ref{thmlocal}}, we have that 
%$u_{\lambda_n,2}\to u$ in $C^1_0(\overline\Omega)$ with $u$ solution of $(P_0)$. As $\lambda_1<\overline\lambda$, $u_{\lambda_1,1}$ can be used to define ${\mathcal S}$ and 
%hence, by construction, $\max\{u_{\lambda_n,2}-u_{\lambda_1,1}\}>0$. This implies that $u\not=u_0$ which contradicts the uniqueness of the solution of $(P_0)$.
\end{proof}

%\begin{remark}
%In case $(P_0)$ has a solution $u_0$ with $cu_0 =0$, it is a solution of $(P_{\lambda})$ for any $\lambda \in \R$.
%\end{remark}

\begin{prop}
\label{unique}
Under assumption $(A)$, assume that $(P_{0})$ has a solution $u_0\leq0$ with $c u_0 \lneqq0.$  Then, for all $\lambda \geq 0$,
 problem $(P_{\lambda})$ has at most one  solution $u\leq 0$.
\end{prop}

\begin{proof}

\noindent The proof is divided in three steps.
\smallbreak

\noindent{\it Step 1: If $u$ is a lower solution of $(P_{\lambda})$ with $u\leq 0$, then $u\ll u_0$.}
In fact, $u$ is a lower solution of $(P_{0})$ and, by Proposition \ref{comparison}, we have $u\leq u_0$. In addition  for $w=u_0-u$, as $c u \leq cu_0 \lneqq0$  we have
$$
\begin{array}{cl}
-\Delta w - \mu(x)\langle \nabla u + \nabla u_0\mid \nabla w\rangle= - \lambda c(x) u\gneqq0,&\mbox{in }\Omega,
\\
w=0,&\mbox{on }\partial\Omega.
\end{array}
$$
This implies that $w\gg0$ i.e. $u\ll u_0\leq0$. 
\medskip

\noindent
{\it Step 2: If we have  two solutions $u_1\leq 0$ and $u_2\leq 0$ of $(P_{\lambda})$ then we have two ordered solutions  $\tilde u_1 \lneqq \tilde u_2\leq u_0$.}
By Step 1, we have  $u_1\ll u_0$ and $u_2\ll u_0$. In case $u_1$ and $u_2$ are not ordered, as $u_0$ is an upper solution of $(P_{\lambda})$,
applying Theorem \ref{sousDe}, 
there exists  a solution $u_3$ of $(P_{\lambda})$ with $\max\{u_1,u_2\}\leq u_3 \leq u_0$.  This proves the step by choosing $\tilde u_1 =u_1$ and $\tilde u_2=u_3$. 
\medskip

\noindent
{\it Step 3: Conclusion.}
Let us assume by contradiction that we have two solutions $u_1\leq 0$ and $u_2\leq 0$. By Step 2, we can suppose  $u_1 \lneqq u_2$. As $|u_2|\gg 0$, the set $\{v\in C^1_0(\overline\Omega)\mid v\leq  |u_2|\}$ is an open neighborhood of $0$ and hence the set  $\{\varepsilon>0\mid u_2-u_1\leq \varepsilon |u_2|\}$ is not empty. Then defining
 $$
\bar\varepsilon=\inf\{\varepsilon>0\mid u_2-u_1\leq \varepsilon |u_2|\}
$$
 we have that $0<\bar\varepsilon <\infty$ and 
\begin{equation}
\label{eq min}
\bar\varepsilon=\min\{\varepsilon>0\mid u_2-u_1\leq \varepsilon |u_2|\}.
\end{equation}
Letting
 $$
w_{\bar\varepsilon}=\frac{(1+\bar\varepsilon)u_2-u_1}{\bar\varepsilon}
$$
 we can write
$$
\nabla u_2 =(\frac{\bar\varepsilon}{1+\bar\varepsilon}) \nabla w_{\bar\varepsilon} + (\frac{1}{1+\bar\varepsilon})\nabla u_1,
$$
and by  convexity 
$$
|\nabla u_2|^2 \leq (\frac{\bar\varepsilon}{1+\bar\varepsilon})| \nabla w_{\bar\varepsilon}|^2+ (\frac{1}{1+\bar\varepsilon})|\nabla u_1|^2.
$$
We then obtain
$$
-\Delta w_{\bar\varepsilon}\leq \lambda c(x) w_{\bar\varepsilon} + \mu(x) |\nabla w_{\bar\varepsilon}|^2+h(x).
$$
By the choice of ${\bar\varepsilon}>0$, $w_{\bar\varepsilon}\leq 0$ and, by Step 1,  $w_{\bar\varepsilon}\ll u_0\leq 0$. 
At this point, we have a contradiction with the definition of $\overline\varepsilon$ given in \eqref{eq min}.
\end{proof}

Our next result can be viewed as a generalization of \cite[Theorem 3.12]{AbPePr}.

\begin{cor}
\label{improve}
Under assumption $(A)$, assume that $h\lneqq 0$. Then, for all $\lambda>0$,
 the problem $(P_{\lambda})$ has exactly one  solution $u\leq 0$.
\end{cor}

\begin{proof}
Clearly  $u \equiv 0$ is an upper solution of $(P_\lambda)$ for all $\lambda\geq 0$.
By Lemma \ref{lem lower}, for all $\lambda\geq 0$, $(P_\lambda)$ 
has a lower solution $\alpha_{\lambda}\leq 0$. From Theorem  \ref{sousDe} it follows that $(P_\lambda)$
has a solution $u_\lambda$ with $\alpha_{\lambda}\leq u_{\lambda}\leq 0$. Now, as $u_0$ satisfies 
$$
- \Delta u_0 = \mu(x) |\nabla u_0|^2 + h(x),
$$
the strong maximum principle and $h\lneqq 0$, implies that  $u_0\ll0$ and in particular $cu_0 \lneqq 0$. We now conclude with Proposition 
\ref{unique}.
\end{proof}

\begin{proof}[Proof of Theorem \ref{thm 2}]

We proceed in several steps. \medskip

\noindent{\it Step 1: For all $\lambda>0$, $u_0$ is a strict upper solution of $(P_{\lambda})$.}
Clearly $u_0$ is an upper solution of $(P_{\lambda})$ which is not a solution. To prove that it is a strict upper solution, we argue 
as in Step 4 of the proof of Theorem  \ref{thmlocal}.
%let $u$ be a solution of $(P_{\lambda})$ with $u\leq u_{0}$. Then  $v:=u_0-u$ satisfies
%$$
%\begin{array}{cl}
%-\Delta v - \mu \langle \nabla u_{0}+\nabla u\mid \nabla v\rangle\geq \lambda c v\geq 0,&\mbox{in }\Omega,
%\\
% v\geq 0,&\mbox{in }\Omega.
%\end{array}
%$$
%By the strong maximum principle we deduce that either $v\gg0$ or $v\equiv0$. If $v\equiv0$, i.e. $u\equiv u_0$, then 
%$c\lambda u_{0}=0$ which contradicts the assumption that  $c u_0 \lneqq 0$.
\medbreak

\noindent{\it Step 2: For all $\lambda>0$, $(P_{\lambda})$ has a strict lower  solution $\alpha$ with $\alpha\leq \beta$ for all  upper solution $\beta$ of $(P_{\lambda})$.}
This is Lemma \ref{lem lower}.
\medbreak

\noindent{\it Step 3: For all $\lambda>0$, $(P_{\lambda})$ has at least two   solutions with }
$$
u_{\lambda,1}\ll u_0, \quad u_{\lambda,1}\ll u_{\lambda,2} \quad \mbox{ and }\quad \max u_{\lambda,2}>0.
$$
By Steps 1, 2 and  Theorem \ref{sousDe}, there exists a $R>0$ such  that $\deg(I-{\mathcal M}_{\lambda}, {\mathcal S}) = 1$ 
with
$$
{\mathcal S} =
\{u \in C^1_0(\overline\Omega) \mid \alpha \ll u \ll u_0,  \|u\|_{C^1} <R\}.
$$
In particular the existence of a first solution  $u_{\lambda,1}\ll u_0$ is proved.
\medbreak

The proof of the existence of a second solution $u_{\lambda,2}$ with $u_{\lambda,1}\ll u_{\lambda,2}$  is derived exactly as in 
Step 3 and 4 of the proof of Theorem \ref{thmlocal}.
%
%Also, using Lemma \ref{non ex}, there exists $A_1 >0$ large enough such that \eqref{EE1a} has no solution for $a\geq A_1$.
%By Theorem \ref{bounds1} and Lemma \ref{lem 2} there exists a  $R_0 >R >0$ such that, for all $a\in [0,A_1]$, 
%every solution of \eqref{EE1a}  satisfies  $\|u\|_{C^1}<R_0$. Hence, by homotopy invariance of the degree, we have
%$$
%\deg(I-{\mathcal M}_{\lambda}, B(0,R_0)) = \deg(I-{\mathcal M}_{\lambda}-{\mathcal L}^{-1}(A_1c), B(0,R_0)).
%$$
%As for $a=A_1$, the problem \eqref{EE1a} has no solution, we have
%$$
%\deg(I-{\mathcal M}_{\lambda}-{\mathcal L}^{-1}(A_1c),B(0,R_0))=0.
%$$
%We then conclude that
%$$
%\deg(I-{\mathcal M}_{\lambda},B(0,R_0) \setminus {\mathcal S})=\deg(I-{\mathcal M}_{\lambda},B(0,R_0))-\deg(I-{\mathcal M}_{\lambda},{\mathcal S})=-1.
%$$
%This proves the existence of a second solution $u_{\lambda,2}$ of $(P_{\lambda})$.
%
By Proposition \ref{unique}, we have $\max u_{\lambda,2}>0$.
\medbreak

%As $\alpha\leq u$ for all $u$ solution of $(P_{\lambda})$, we can choose $u_{\lambda,1}$ as the minimal solution with $u_{\lambda,1}\geq \alpha$. Hence we have 
%$u_{\lambda,1}\lneqq u_{\lambda,2}$ as otherwise  there exists a solution $u$ with $\alpha\leq u \leq \min(u_{\lambda,1},u_{\lambda,2})$ which contradicts the 
%minimality of $u_{\lambda,1}$.
%\medbreak

\noindent{\it Step 4: If $\lambda_1< \lambda_2$, then $u_{\lambda_1,1}\gg u_{\lambda_2,1}$.}
As $u_{\lambda_1,1}$ is a strict upper solution of $(P_{\lambda_2})$ and $u_{\lambda_2,1}$ is the minimal solution of $(P_{\lambda_2})$, 
we have $u_{\lambda_1,1}\gg u_{\lambda_2,1}$.
\medbreak

\noindent{\it Step 5: Behaviour of the solutions for $\lambda\to0$.}
This can be proved as  in Step 7 of the proof of Theorem  \ref{thmlocal}.
%Let $\{\lambda_n\} \subset \,]0,+\infty[$ be a decreasing sequence such that $\lambda_n\to 0$. We have $u_{\lambda_n,1}\leq u_{\lambda_{n+1},1}\leq u_0$. 
%As in Step 9 of the proof of Theorem \ref{thm 1}, we have that
%$u_n\to u_0$ in $C^1_0(\overline\Omega)$. 
%
%Now, let us consider the sequence  $\{u_{\lambda_n,2}\}$ in case $u_0\ll 0$. If $\{u_{\lambda_n,2}\}$ is bounded, then as in Step 9  of the proof of Theorem \ref{thm 1}, 
%we obtain again that  $u_n\to u$ in $C^1_0(\overline\Omega)$ with $u$ solution of $(P_0)$. As  $\max u_{\lambda_n,2}>0$ and $u_0\ll 0$, 
%it follows that $u\not=u_0$ which contradicts the uniqueness of the solution of $(P_0)$.
\end{proof}

%\begin{remark}
%The Proof of Corollary \ref{improve} shows that $(P_0)$ has a solution $u_0 \leq 0$ with $cu_0 \lneqq 0$ if $h\lneqq 0$.
%\end{remark}

\begin{proof}[Proof of Corollary  \ref{Cor Negatif}]
By the proof of Corollary \ref{improve}, as $h\lneqq 0$, we have the existence of a solution $u_0$ of  $(P_0)$ with $u_0\ll 0$ and hence the result follows by Theorem \ref{thm 2}.
\end{proof}

%\begin{remark}
%The Proof of Corollary \ref{improve} shows that $(P_0)$ has a solution $u_0 \leq 0$ with $cu_0 \lneqq 0$ if $h\lneqq 0$.
%\end{remark}

\begin{proof}[Proof of Theorem \ref{thm 3}]
First observe that if $(P_{\lambda})$ has an upper solution $\beta_{\lambda}\leq0$, then  $\beta_{\lambda}$ satisfies also $c \beta_{\lambda}\lneqq 0$ as otherwise, it is also an upper solution of $(P_0)$, which contradicts the assumption (a) by Lemma  \ref{lem lower} and Theorem \ref{sousDe}.

Let us define 
$$
\underline\lambda=\inf\{\lambda \geq 0 \mid (P_{\lambda}) \mbox{ has an upper solution } \beta_{\lambda}\leq0 \mbox{ with } c \beta_{\lambda}\lneqq 0 \}.
$$

Let $\lambda>\underline\lambda$. By definition of $\underline\lambda$, there exists $\tilde\lambda\in\,]\underline\lambda,\lambda[$ such that 
$(P_{\tilde\lambda})$  has an upper solution $\beta_{\tilde\lambda}\leq0$ with $c \beta_{\tilde\lambda}\lneqq 0 $. 
Clearly $\beta_{\tilde\lambda}$ is an upper solution of $(P_{\lambda})$ which is not a solution and hence, as in Step 4 of the proof of Theorem  \ref{thmlocal}, 
$\beta$ is a strict upper solution of $(P_{\lambda})$. 
%It is a strict upper solution as if $u$ is a solution of $(P_{\lambda})$ with $u\leq \beta_{\tilde\lambda}$ then $w=\beta_{\tilde\lambda}-u$ satisfies
%$$
%\begin{array}{cl}
%-\Delta w - \mu(x) \langle \nabla \beta_{\tilde\lambda}+\nabla u\mid \nabla w\rangle\geq \tilde\lambda c(x) \beta_{\tilde\lambda} - \lambda c(x) u \gneqq 0,&\mbox{in }\Omega,
%\\
% w\geq 0,&\mbox{in }\Omega.
%\end{array}
%$$
%Hence, by the strong maximum principle  we have $w\gg0$.
%\medbreak

By Lemma \ref{lem lower}, $(P_{\lambda})$ has a strict lower  solution $\alpha\leq \beta_{\tilde\lambda}$ and $\alpha\leq u$ for all solution $u$ of $(P_{\lambda})$.
Using Theorem \ref{sousDe} there exists $R>0$ such that  $\deg(I-{\mathcal M}_{\lambda}, {\mathcal S}) = 1$ with
$$
{\mathcal S} =
\{u \in C^1_0(\overline\Omega) \mid \alpha \ll u \ll \beta_{\tilde\lambda}, \|u\|_{C^1} <R \}.
$$
In particular the existence of a first solution $u_{\lambda,1}\ll 0$ follows. 
\medbreak

To obtain a second solution $u_{\lambda, 2}$ satisfying $u_{\lambda,1} \ll u_{\lambda,2}$ we now just repeat the arguments of Steps 3 and 4 of the proof of 
Theorem \ref{thmlocal}.  
\medskip

Again, following the arguments of Step 4 of the proof of Theorem \ref{thm 2}, we prove that if $\lambda_1< \lambda_2$, then $u_{\lambda_1,1}\gg u_{\lambda_2,1}$.
\medbreak

To show that $(P_{\underline\lambda})$ has at least one solution with $u \leq 0$, let $\{\lambda_n\} \subset \, ]\underline\lambda, +\infty[$ be a decreasing sequence such that 
$\lambda_n \to  \underline\lambda$ 
and $\{u_n \} \subset W^{2,p}(\Omega)$ be a sequence of corresponding  solutions with $u_{n}\leq u_{n+1}\leq 0$. 
%By Theorem \ref{bounds1} 
As $\{u_n\}$ is increasing and bounded above, there exists  $M>0$ such that, for all $n \in \mathbb N$, $\|u_n\|_{\infty}<M$ and hence, 
arguing as in Step 5 of the proof of 
Theorem \ref{thmlocal}, we prove that $(P_{\underline\lambda})$ has at least one solution with $u \leq 0$.
%by Lemma \ref{lem 2}, there exists $R>0$ such that, for all $n \in \mathbb N$,  $\|u_n\|_{W^{2,p}}<R$.
%Hence, up to a subsequence $u_n\to u$ in $C^1_0(\overline\Omega)$. 
%From this strong convergence the result follows easily.
\medbreak

By assumption (a), we have that   ${\underline\lambda}>0$ as we just proved that $(P_{\underline\lambda})$ has at least one solution with $u \leq 0$.
The proof of the uniqueness of the non-positive solution of $(P_{\underline\lambda})$ follows then 
as in Step 6 of the proof of Theorem \ref{thmlocal}. Finally (iii) follows by definition of  
${\underline\lambda}>0$ and the first part of the proof.
\end{proof}

\begin{proof}[Proof of Theorem \ref{thm 4}]
Let $\lambda>\nu_1$. We proceed in several steps.
\medskip

\noindent{\it Step 1: For $k>0$ small, $(Q_{\lambda,k})$ admits a solution.}
In view of Lemma \ref{lem lower} and of Theorem  \ref{sousDe} it suffices to show that $(Q_{\lambda,k})$ admits an upper solution.

Let $\varepsilon_0 >0$ be given by Proposition \ref{anti} corresponding to $\bar c = c$, $\bar d = \mu_2 \tilde h^-$ and $\bar h = \mu_2 \tilde h^+$ 
and choose $\lambda_0 \in \,]\nu_1, \min(\nu_1+\varepsilon_0, \nu_1 + \frac{\lambda - \nu_1}{2})]$. Then let $w\ll0$  be the solution of
$$
\begin{array}{cl}
-\Delta u+\mu_2 \tilde h^- u = \lambda_0 c u+ \mu_2 \tilde h^+, &\mbox{ in }\Omega,
\\
u=0,  &\mbox{ on }\partial \Omega.
\end{array}
$$

Also taking $\delta >0$ small enough we have that
$$
\lambda_0 s \geq (1+\lambda s)\ln(1+\lambda s)
$$
for all $s \in [- \delta, 0]$. Thus defining  
$\tilde\beta_k=\frac{k}{\lambda} w$ for $k>0$ small enough, it follows  that $\tilde\beta_k \in [-\delta, 0]$ and 
$$
\begin{array}{cl}
%-\Delta \tilde\beta_k-\mu_2 h \tilde\beta_k  \geq
-\Delta \tilde\beta_k+\mu_2 \tilde h^- \tilde\beta_k  
\geq  
c (1+\lambda \tilde\beta_k)\ln(1+\lambda \tilde\beta_k) + k\frac{\mu_2}{\lambda} \tilde h^+, &\mbox{ in }\Omega,
\\
\tilde\beta_k=0,  &\mbox{ on }\partial \Omega.
\end{array}
$$
At this point defining  $\beta_k=\frac{1}{\mu_2}\ln(\lambda \tilde\beta_k +1)$ we see, after some standard calculations, 
that $\beta_k\ll0$ is an upper solution for $(Q_{\lambda,k})$.
\medbreak

\noindent{\it Step 2: For $k$ large, the problem $(Q_{\lambda,k})$ has no solution.}
Let $\phi\in C^{\infty}_0(\Omega)$ such that $\phi^2\gg0$. Then, using $\phi^2$ as test function we obtain,
by Lemma \ref{lem 4.1},
$$
\begin{array}{rcl}
\displaystyle
\int_{\Omega} \frac{1}{\mu(x)}|\nabla \phi|^2 \,dx
&\geq &
\displaystyle
2\int_{\Omega} \phi \langle \nabla u, \nabla \phi\rangle \,dx - \int_{\Omega} \mu(x) |\nabla u|^2 \phi^2 \,dx
\\[3mm]
&=&
\displaystyle
\lambda \int_{\Omega} c u \phi^2  \,dx + k \int_{\Omega} \tilde h^+\phi^2 \,dx -\int_{\Omega} \tilde h^- \phi^2 \,dx
\\[3mm]
&\geq &
\displaystyle
-\lambda M \int_{\Omega} c  \phi^2 \,dx + k\int_{\Omega} \tilde h^+\phi^2 \,dx -\int_{\Omega} \tilde h^- \phi^2 \,dx
\end{array}
$$
which is a contradiction for $k>0$ large enough.
\medbreak

\noindent{\it Step 3: Define 
$$
\overline k=\sup\{ k\in \,]0,+\infty[ \, \mid \mbox{the problem } (Q_{\lambda,k})\mbox{ has  at least one solution}\},
$$ 
then $\overline k\in \, ]0,+\infty[$ and for $k\in\, ]0,\overline k[$, the problem $(Q_{\lambda,k})$ has a strict upper solution.}
%\medskip

By Step 1 and 2 we have  easily $\overline k\in \,]0,+\infty[$.

Let $k\in\,]0,\overline k[$ and $\tilde k\in \,]k,\overline k[$ be such that $(Q_{\lambda,\tilde k})$ has a solution $\tilde \beta$. Then $\beta=\frac{k}{\tilde k} \tilde \beta$ is an upper solution of $(Q_{\lambda, k})$ as 
$$
\begin{array}{cl}
\begin{array}[b]{rcll}
- \Delta \beta &=&  \lambda c(x)\beta+\frac{\tilde k}{k} \mu(x) |\nabla \beta|^2 + k\tilde h^+(x)-\frac{k}{\tilde k} \tilde h^-(x)
\\[3mm]
&\geq& \lambda c(x)\beta+ \mu(x) |\nabla \beta|^2 + k\tilde h^+(x)-\tilde h^-(x), 
\end{array}&\mbox{in }\Omega,
\\[3mm]
\beta\geq0,&\mbox{on } \partial\Omega,
\end{array}
$$
i.e. $\beta$ is an upper solution of $(Q_{\lambda,k})$. Now, as in Step 4 of the proof of Theorem \ref{thm 1} we can prove that $\beta$ is a strict upper solution of $(Q_{\lambda,k})$. 
\medbreak

\noindent{\it Step 4: Conclusion. }
At this point the proof follows as in the proof of Theorems \ref{thm 1} or \ref{thm 2}.  This is possible in view of Step 2 and of Theorem \ref{bounds1}. 
\end{proof}

\begin{proof}[Proof of  Corollary \ref{surprise}] 
First observe that, by \cite[Lemma 6.1]{ArDeJeTa} (see also the proof of Corollary \ref{Cor 4.1} above), we know that  $(P_{\gamma_1})$ has no solution. Hence also, for all $\lambda>0$,  $(P_{\lambda})$ has no solution with $cu\equiv 0$ as otherwise $u$ is solution for every $\lambda\in \mathbb R$ which contradicts the non existence of a solution for $\lambda=\gamma_1$.
\medbreak

By Step 3 of the proof of Theorem \ref{thm 4}, there exists $\tilde k >0$ such that, for all $k\in \,]0,\tilde k]$,  
the problem $(P_{\tilde \lambda})$ has a strict upper solution $\beta_0$ with $\beta_0\ll0$. The existence of $\lambda_2>\gamma_1$ as in (iii) can then be deduced from Theorem \ref{thm 3}.
\medbreak

By \cite[Theorem 1.1]{ArDeJeTa}, decreasing $\tilde k$ if necessary, we know that for all $k\in \,]0,\tilde k]$, the problem $(P_0)$ has a solution $u_0\gg0$. 
Hence the existence of $\lambda_1$ as in (i) can be deduced from  Theorem \ref{thm 1}. 
\end{proof}

\begin{proof}[Proof of  Theorem \ref{cas h0}]
First observe that, for all $\lambda\in \mathbb R$, $u\equiv 0$ is solution of \eqref{eq 0}.

\noindent{\it Step 1: for all $\lambda\in\, ]0, \gamma_1[$, the problem \eqref{eq 0}
has a second solution $u_{\lambda,2}\gneqq 0$. }
Let us prove that the problem \eqref{eq 0} has a strict upper solution $\beta\gg0$. 
To this end, let $\lambda<\gamma_1$ and $\varepsilon>0$ such that, 
for all $v\in [0,\varepsilon]$, $\lambda \frac{(1+\mu_2 v)\ln(1+\mu_2 v)}{\mu_2}\leq \gamma_1 v$.
Consider then the function $\tilde\beta=\varepsilon\varphi_1$ where $\varphi_1$ denotes the first eigenfunction of \eqref{eigenvaluep} with $\|\varphi_1\|_{\infty}=1$ and observe that
$$
\begin{array}{cl}
\displaystyle
-\Delta \tilde\beta \gneqq  \lambda c(x)\frac{(1+\mu_2\tilde\beta)\ln(1+\mu_2\tilde\beta)}{\mu_2}, &\mbox{a.e. in }\Omega,
\\
\tilde\beta =0,  &\mbox{on }\partial \Omega.
\end{array}
$$
Hence for $\beta$ being defined by 
$\beta= \frac{\ln(\mu_2\tilde\beta+1)}{\mu_2}$, we have
$$
\begin{array}{cl}
\displaystyle
    -\Delta \beta \gneqq  \lambda c(x)\beta + \mu_2 |\nabla \beta|^2\geq  \lambda c(x)\beta + \mu(x) |\nabla \beta|^2, &\mbox{a.e. in }\Omega,
\\
\beta =0,  &\mbox{on }\partial \Omega.
\end{array}
$$
This implies, as in Step 4 of the proof of  Theorem \ref{thmlocal}, that $\beta\gg0$ is a strict upper solution of \eqref{eq 0}.

By \cite[Lemma 6.1]{ArDeJeTa}, we know that, every solution $u$ of \eqref{eq 0} satisfies $u\geq 0$ and by Lemma \ref{lem lower}, the problem  \eqref{eq 0} has a strict lower solution $\alpha\lneqq 0$. Hence we conclude the proof of (i) following the same lines as in the proof of Theorem \ref{thm 1}, 
the solution $u_{\lambda,1}$ being $u\equiv 0$. 
\medbreak

\noindent{\it Step 2: For  $\lambda=\gamma_1$ the problem \eqref{eq 0}  has only the trivial solution. }
This can be proved as in Corollary \ref{Cor 4.1}.
\medbreak

\noindent{\it Step 3: For $\lambda>\gamma_1$, the problem \eqref{eq 0} has a second solution $u_{\lambda,2}\ll 0$. }
Let $\lambda>\gamma_1$ and $\lambda_0\in \,]\gamma_1, \lambda]$ such that, by Proposition \ref{anti}, the problem 
\begin{equation}
\label{eq 0AM}
- \Delta u = \lambda_0 c(x)u+ 1, \quad u \in H^1_0(\Omega) \cap L^{\infty}(\Omega)
\end{equation}
has a solution $u\ll0$.
This implies  that for $\varepsilon>0$ small enough, the function $\beta_0=\varepsilon u$ satisfies 
$$
- \Delta \beta_0= \lambda_0 c(x)\beta_0+ \varepsilon\geq \lambda_0  c(x)\beta_0 + \mu \varepsilon^2|\nabla u|^2= \lambda_0  c(x)\beta_0 + \mu |\nabla \beta_0|^2
$$
and the problem  $(P_{\lambda_0})$ has an upper solution $\beta_0$  with $\beta_0\leq0$ and $c \beta_0 \lneqq  0$. The result follows by Theorem \ref{thm 3}.
\end{proof}

\section{Complement in case $\mu$ constant}
\label{Section4}

First observe that, in the case $\mu$ constant, we have a necessary and sufficient condition for the existence of a solution of  $(P_0)$.

\begin{prop}
\label{Prop 5.1}
Assume that $(A)$ holds with $\mu$ a positive constant. Then  $(P_0)$  has a solution if and only if
the first eigenvalue $\xi_1(c)$ of the problem
$$
\begin{array}{cl}
-\Delta w-\mu h w = \xi\, c \,  w, &\mbox{ in }\Omega,
\\
w=0,  &\mbox{ on }\partial \Omega,
\end{array}
$$
satisfies $\xi_1(c)>0$.
\end{prop}
 
\begin{proof}
By  \cite[Remark 3.2]{ArDeJeTa}, we know that, if $(P_0)$  has a solution then 
$$
%\begin{equation}
%\label{cond ArDeJeTa}
\displaystyle \inf_{ \{ u \in H^1_0(\Omega) \mid \,\|u\|_{H^1_0(\Omega)}=1 \}} \,  \displaystyle \int_{\Omega} \left(|\nabla u|^2 -  \mu h(x) u^2 \right) dx >0,
$$
%\end{equation}
and hence $\xi_1(c)>0$.
\medbreak

On the other hand, if $\xi_1(c)>0$, this implies that the problem
$$
\begin{array}{cl}
-\Delta w-\mu h w = \mu h^+, &\mbox{ in }\Omega,
\\
w=0,  &\mbox{ on }\partial \Omega,
\end{array}
$$
has a positive solution $w$. It is then easy to prove that $\beta=\frac{1}{\mu} \ln(w+1)$ is a positive upper solution of $(P_0)$. As, by Lemma \ref{lem lower}, 
 $(P_0)$ has a lower solution $\alpha\leq \beta$, we conclude by application of Theorem \ref{sousDe}.
\end{proof}

\begin{prop}
\label{justtwo}
Assume that $(A)$ holds with  $\mu$ a positive constant and that there exists a sequence $\{\lambda_n\} \subset \, ]0, + \infty[$ with $\lambda_n \to0$ and two sequences 
$\{u_{\lambda_n}\}$, $\{\tilde{u}_{\lambda_n}\}$ of solutions of $(P_{\lambda_n})$ such that
$$ 
\lambda_n\|u_{\lambda_n}\|_{\infty} \to 0 \quad \mbox{ and } \quad \lambda_n \|\tilde{u}_{\lambda_n}\|_{\infty} \to 0,
$$
as $\lambda_n \to 0$. Then, for any $n \in \N$ sufficiently large, $u_{\lambda_n}=\tilde{u}_{\lambda_n}.$
\end{prop}

%that $(P_0)$ has a solution and that there exists a function $\lambda \to D(\lambda)$ with 
%$D(\lambda)\lambda \to 0$  as $ \lambda \to 0^+$ such  that, for any solution $u$ of \eqref{EE1a} one has
%\begin{equation}
%\label{belowgrowth}
%\|u\|_{\infty} \leq D(\lambda).
%\end{equation}
%Then there exists a $\lambda_0 >0$ such that, for all $\lambda \in \, [0, \lambda_0[$, the problem $(P_{\lambda})$ has at most one solution.
%\end{prop}

\begin{proof}
If $u_n$ is a solution of $(P_{\lambda_n})$ by the change of variable $u_n=\frac{1}{\mu}\ln(v_n +1)$ we have that $v_n>-1$ is solution of
\begin{equation}
\label{mod}
\begin{array}{cl}
-\Delta v_n-\mu h v_n = \lambda_n \, c \, (1+  v_n)\ln(1+  v_n) + \mu h, &\mbox{ in }\Omega,
\\
v_n=0,  &\mbox{ on }\partial \Omega.
\end{array}
\end{equation}
Setting $D(\lambda_n) := \|u_n\|_{\infty}$, since $v_n = e^{\mu u_n}-1$ we deduce that $\|v_n\|_{\infty} \leq C(\lambda_n)$ where
$$
C(\lambda_n) = e^{\mu D(\lambda_n)}-1.
$$
Now observe that if we assume that $\lambda_n D(\lambda_n) \to0$ then
$$
\lim_{\lambda_n\to0}\lambda_n (\ln(1+ C(\lambda_n)) +1)  = \lim_{\lambda_n\to0} \lambda_n D(\lambda_n) =0.
$$
As, by Proposition \ref{Prop 5.1}, $\xi_1(c)>0$, there exists $n_0 \in \N$ such that, for all $n \geq n_0$  
$$
\lambda_n (\ln(1+  C(\lambda_n)) +1) < \xi_1(c).
$$

If we assume by contradiction that, for $n \geq n_0$, $u_{\lambda_n}\not=\tilde{u}_{\lambda_n}$ then \eqref{mod} has also two distinct solutions 
$v_{n,1}$ and $v_{n,2}$ and  $w_n=v_{n,1}-v_{n,2}$ is a solution of
\begin{equation}
\label{most}
\begin{array}{cl}
-\Delta w-\mu h w = \lambda_n\, c \, \rho_n(x) \, w, &\mbox{ in }\Omega,
\\
w=0,  &\mbox{ on }\partial \Omega
\end{array}
\end{equation}
with
$$
\begin{array}{rcl}
\rho_n(x)&=&
\displaystyle
\frac{(1+ v_{n,1}(x))\ln(1+ v_{n,1}(x)) -(1+ v_{n,2}(x))\ln(1+ v_{n,2}(x))}{v_{n,1}(x) - v_{n,2}(x)} ,\hspace{4mm}\mbox{}
\\
&&\hspace{84mm} 
\mbox{if } v_{n,1}(x)\not= v_{n,2}(x),
\\
&=& \ln(1+ v_{n,1}(x)) +1,\hfill\mbox{if } v_{n,1}(x)= v_{n,2}(x),
\end{array}
$$
and by assumption $0<\lambda_n \rho_n<\xi_1(c)$.
\medbreak

As \eqref{most} has a nontrivial solution,  we have $\xi_i(\lambda_n c\rho_n)=1$ for some $i\in \mathbb N$.
%, where $\mu_i(c\rho)$ are the eigenvalues of
%$$
%\begin{array}{cl}
%-\Delta w-\mu h w = \mu_i(\rho) c \, \rho(x) w, &\mbox{ in }\Omega,
%\\
%w=0,  &\mbox{ on }\partial \Omega.
%\end{array}
%$$
Moreover, as $\lambda_n \rho_n<\xi_1(c)$,  we know by \cite{DG} that $ 1 = \xi_i(\lambda_n c\rho_n)>\xi_i(c\xi_1(c))=\xi_i(c)/\xi_1(c)$. 
This contradicts that the sequence of eigenvalues $(\xi_i(c))_i$ is strictly increasing and proves the proposition.
\end{proof}

Under the assumption that $\mu$ is constant, the following  lemma gives informations on the set of solutions of $(P_{\lambda})$ for $\lambda >0$ small.

\begin{cor}
Assume that assumption $(A)$ holds with  $\mu$ a positive constant and that $(P_0)$ has a solution $u_0$. Let $\{u_{\lambda_n}\}$ be a sequence of solutions of $(P_{\lambda_n})$  satisfying 
$ \lambda_n\|u_{\lambda_n}\|_{\infty} \to 0$
as $\lambda_n \to 0$. Then we have, for any $n \in \N$ sufficiently large,
\begin{enumerate}
\item[(i)] $u_{\lambda_n}= u_{\lambda_n,1}$ where $u_{\lambda_n, 1}$ is the minimal solution given in Theorem \ref{thmlocal}. 
In particular $u_{\lambda_n} \to u_0$ in $C_0^1(\overline{\Omega})$.
\smallskip
\item[(ii)] $(\lambda_n, u_{\lambda_n})$ belongs to ${\mathcal C}^+$  where ${\mathcal C}^+$ is defined in Theorem \ref{ADJT1}.
\end{enumerate}
\end{cor}

\begin{proof}
Since $\{u_{\lambda_n, 1}\}$ satisfies $\lambda_n \|u_{\lambda_n, 1}\|_{\infty} \to 0$ as $\lambda_n \to 0$ it directly follows from Proposition \ref{justtwo} that, for any $n \in \N$ large enough, $u_{\lambda_n} = u_{\lambda_n, 1}$. In particular it follows from Theorem \ref{thmlocal} that $u_{\lambda_n} \to u_0$ in $C_0^1(\overline{\Omega})$. Now by Theorem \ref{ADJT1} we know that, for $n \in \N$ large enough, there exists $\overline{u}_{\lambda_n}$ such that 
$(\lambda_n, \overline{u}_{\lambda_n}) \in {\mathcal C}^+$.  Since, by continuity, we have that $\lambda_n \|\overline{u}_{\lambda_n}\|_{\infty} \to 0$ we deduce by (i) that $\overline{u}_{\lambda_n}= u_{\lambda_n, 1}$. Thus $u_{\lambda_n} = \overline{u}_{\lambda_n} $. 
\end{proof}

Also using again that $\lambda_n \|u_{\lambda_n, 1}\|_{\infty} \to 0$ as $\lambda_n \to 0$, we immediately deduce from Proposition \ref{justtwo} the following result.

\begin{cor} 
Assume that $(A)$ holds with $\mu$ a positive constant.
\begin{enumerate}
\item[(i)] In Theorems \ref{thmlocal}, \ref{thm 1},  \ref{thm 2} and \ref{cas h0} we have that
$$
\liminf_{\lambda \to 0^+} \lambda \|u_{\lambda, 2}\|_{\infty} >0.
$$
\item[(ii)] The bound derived in Theorem \ref{bounds1}, $\|u\|_{\infty} \leq M(\lambda)$ for any solution $u$ of \eqref{EE1a} with 
$$
\limsup_{\lambda \to 0^+}M(\lambda)\lambda \leq C
$$ 
for some $C>0$ is sharp.
\end{enumerate}
\end{cor}

%\begin{proof}
%From the uniqueness of the solution $u_0$ of $(P_0)$ and Proposition 5.2 one easily deduce (i). Now 
%(ii) directly follows from (i).
%\end{proof}

\section{Case $N=1$ and open problems}
\label{Section5}

In case $\Omega=[-\frac{T}{2}, \frac{T}{2}]$ i.e. $N=1$ and  $\mu>0$, $c>0$ and $h\not=0$ are constants, we can make a more precise study of the situation.

By the classical change of variable $v=e^{\mu u}-1$, we are reduce to the problem 
\begin{equation}
\label{eq time map}
\begin{array}{c}
-v''-\mu h v=\lambda (v+1) \ln(v+1)+\mu h, \quad\mbox{ in } [-\frac{T}{2}, \frac{T}{2}]
\\[2mm]
v>-1, \quad \mbox{ in } [-\frac{T}{2}, \frac{T}{2}]
\\[2mm]
v(-\frac{T}{2})=0, \quad v(\frac{T}{2})=0.
\end{array}
\end{equation}
It is easy to prove that in case $\lambda=0$  this problem has a solution if and only if $\mu \, h<(\pi/T)^2$ which corresponds to the condition \eqref{cond ArDeJeTa}.

As this problem is autonomous, we can make a phase-plane analysis. There are three different situations: $h>0$ and $\lambda>0$ small, $h>0$ and $\lambda$ large, $h<0$.
\medbreak

\noindent{\bf Case 1: $0<\lambda<2\mu h$.}
In that case the phase plane is given by
\begin{center}
\begin{tikzpicture}[x=15ex, y=10ex]
  \draw[->] (-1.1,0) -- (0.6, 0) node[below]{$v$};
  \draw[->] (0,-1.4) -- (0, 1.4) node[left]{$v'$};
  \begin{scope}[color=red]
    \input{phase-plane_m1_l1.tex}    
  \end{scope}
  \draw[fill] (-1,0) circle(1pt) node[above left]{$-1$};
\end{tikzpicture}
\end{center}
We then see that the only possibility is to have positive solutions. Moreover considering the time map $T_+(a)$ which gives the time for the positive part of the orbit to go 
from $(0,a)$  to $(0,-a)$ with $a>0$, it is easy to prove that 
$$
\lim_{a\to0} T_+(a)=0  \quad \mbox{ and } \quad \lim_{a\to+\infty} T_+(a)=0.  
$$
This implies the existence of  $T_0>0$ such that, for all $T<T_0$, the problem \eqref{eq time map} has two solutions and, 
for $T>T_0$ the problem \eqref{eq time map} has no solution. Numerical experiment shows that the count is exact.

This corresponds to what we prove in Theorem \ref{thm 1} together with \cite[Lemma 6.1]{ArDeJeTa} where it is shown that, in case $h\gneqq 0$, for all $\lambda <\gamma_1$, 
every solution of $(P_{\lambda})$ is non-negative.
\smallbreak

{\bf Open problem 1}
Can we prove that, for all $\lambda <\gamma_1$, every solution of $(P_{\lambda})$ is non-negative under the sole condition that $(P_0)$ has a solution $u_0$ with 
$u_0\geq 0$ and $cu_0\gneqq 0$?
\smallbreak

{\bf Open problem 2} 
Can we prove, under the assumptions of Theorem \ref{thm 1} or even under the assumptions of Theorem \ref{thmlocal}, that, for all $\lambda <\gamma_1$, we have at most two solutions? 
\medbreak

%\pagebreak

\noindent{\bf Case 2:  $\lambda>2\mu h>0$.}
In that case the phase plane is richer and is given by
\begin{center}
\begin{tikzpicture}[x=26ex, y=10ex]
  \draw[->] (-0.3,0) -- (0.6, 0) node[below]{$v$};
  \draw[->] (0,-1.4) -- (0, 1.4) node[left]{$v'$};
  \begin{scope}[color=red]
    \input{phase-plane_m1_l4.tex}    
  \end{scope}
  \draw[fill] (-0.25, 0) circle(1pt) node[above left]{$-1$};
\end{tikzpicture}
\end{center}
We see the possibilities of positive solutions but also of negative or sign-changing ones.

We can prove that if $\mu\,h\geq (\pi/T)^2$ or $\lambda\geq (\pi/T)^2$ then the problem \eqref{eq time map} has no  non-negative solutions i.e. the time $T_+(a)$   for the positive part of the orbit to go from $(0,a)$  to $(0,-a)$ with $a>0$ is too short with respect to the length of the interval we consider.

For what concerns negative or sign-changing solutions, we see that, if we denote by $T_0$ the time needed by the solution with $\max_{]-\frac{T}{2}, \frac{T}{2}[} u = 0$ 
to make a turn in the  phase plane, then for $T>T_0$, there is a negative solution as well as a sign-changing one. This is the situation studied in Theorem \ref{thm 3}.

But for $T>kT_0$ we have also solutions making $k$ turns in the phase plane.
\smallbreak

{\bf Open problem 3}
Can we prove in Theorem \ref{thm 3} that the second solution changes sign?
\smallbreak

{\bf Open problem 4}
Can we prove that in a small interval below $\underline \lambda$ in Theorem \ref{thm 3}, the problem $(P_{\lambda})$ has no solution and that 
$u_{\underline\lambda}\leq 0$ but $u_{\underline\lambda}\not\ll 0$ ?
\smallbreak

{\bf Open problem 5}
Can we prove the existence of more then two solutions for $\lambda$ large?   Is there a link with the spectrum of the problem 
\begin{equation}
\label{eigenvaluepp}
%\begin{array}{ccc}
-\Delta \varphi_{1} = \gamma c (x) \varphi_{1},  \quad \varphi_1 \in H^1_0(\Omega) ?
%\\
% \varphi_{1} =0,  &\mbox{ on }& \partial\Omega
%\end{array}
\end{equation}

\pagebreak

\noindent{\bf Case 3:  $h<0$.}
In that case, the phase portrait is given by
\begin{center}
\begin{tikzpicture}[x=26ex, y=10ex]
  \draw[->] (-0.3,0) -- (0.6, 0) node[below]{$v$};
  \draw[->] (0,-1.4) -- (0, 1.4) node[left]{$v'$};
  \begin{scope}[color=red]
    \input{phase-plane_m-1_l4.tex}    
  \end{scope}
  \draw[fill] (-0.25, 0) circle(1pt) node[above left]{$-1$};
\end{tikzpicture}
\end{center}
and we see that we have always a negative solution. Moreover, if we denote by $T_1$ the time needed by the solution with $\min_{]-\frac{T}{2}, \frac{T}{2}[} u = 0$ 
to make a turn in the  phase plane, then, for $T< T_1$ the problem \eqref{eq time map} has a positive solution (as again, considering the time map $T_+(a)$ which gives the time for the positive part of the orbit to go  
from $(0,a)$  to $(0,-a)$ with $a>0$, we have $\displaystyle \lim_{a\to+\infty} T_+(a)=0$)  and for $T>T_1$ we have a sign-changing solution.
This is the situation considered in Theorem \ref{thm 2}.
\smallbreak

{\bf Open problem 6}
Can we prove in Theorem \ref{thm 2} that the second solution is positive for $\lambda>0$ small and changes sign for $\lambda$ large?
\smallbreak

Moreover,  for $T>kT_1$ we have also solutions making $k$ turns in the phase plane.
\smallbreak

{\bf Open problem 7}
As in open problem 5, can we prove the existence of more then two solutions for $\lambda$ large?
\smallbreak

In addition to the above open problems directly induced by the phase plane analysis, we also propose the following questions.
\smallbreak

{\bf Open problem 8}
Can we give a more precise characterization of the situation in case $h$ changes sign or $u_0$  changes sign?
\smallbreak

{\bf Open problem 9}
 In \cite{So} some a priori bounds for non-negative solutions have been derived without assuming that $\mu(x) \geq \mu_1 >0$. 
Can a similar result be obtained in the general case ? 
\smallbreak

{\bf Open problem 10}
 In \cite{ArDeJeTa}, the results are obtained under less regularity assumptions ($c$, $h\in L^p(\Omega)$ with $p>N/2$). In \cite{AbPePr}, the regularity is even weaker. If some of our results are still valid when (A) is weakened, how dependent is the structure of the set of solutions of our regularity assumption ? 
%\medskip

\section{Appendix : Proof of Theorem \ref{sousDe}.} 

\label{appendix}

Let us denote $\alpha := \max\{\alpha_i\mid 1\leq i \leq k\}$ where $\alpha_1, \ldots,  \alpha_k$ are regular lower solutions of \eqref{De-He} and   
$\beta=\min\{\beta_j\mid 1\leq j \leq l\}$ where $\beta_1, \ldots,  \beta_l$ are regular upper solutions of \eqref{De-He}.
The proof is divided into   three parts.

\smallbreak
\noindent
{\it Part 1. Existence of a solution $u$ of
\eqref{De-He} with $\alpha\leq u\leq\beta$.}  
Observe that by Lemma \ref{lem 2},  there
exist   $R>0$ such that, for every function $f$ satisfying \eqref{Hf}
and every solution $u$ of  \eqref{De-He} with
$\alpha\leq u \leq \beta$, we have
\begin{equation}
\label{3-5.2}
\|u\|_{W^{2,p}}<R \qquad \mbox{ and }\qquad \|u\|_{C^1_0}<R.
\end{equation}
\smallbreak

\noindent
\textit{Step 1.  Construction of a modified problem.}
Take $\overline R$ such that  
$$
\overline R>  \max\{R, \max_{1\leq i\leq
k}\|\alpha_i\|_{C^1},
\max_{1\leq j\leq l}\|\beta_j\|_{C^1}\}
$$
and set for a.e. $x \in \Omega$
and every $(s, \xi) \in \mathbb R \times \mathbb R^N$,
$$
\overline f (x,s,\xi) = \left\{
\begin{array}{ll}
f(x,s,\xi),&\mbox{if }|\xi|\leq \overline R,
\\
f(x,s,\overline R\frac{\xi}{|\xi|}),&\mbox{if }|\xi|> \overline R. 
\end{array}
\right.
$$
Now we define the functions 
$$
p_i(x,  s, \xi) =
\left\{
\begin{array}{ll}
\overline f(x,  \alpha_i(x), \xi)
+ \omega_{1,i}(x, \alpha_i(x) - s),
&\mbox{ if } s < \alpha_i(x),
\\
\overline f(x, s, \xi),
&\mbox{ if } s \geq \alpha_i(x),
\end{array}
\right.
$$
where
$$
\omega_{1,i}(x,  \delta) = \max_{|\xi| \leq \delta}
|\overline f(x,  \alpha_i(x), \nabla\alpha_i(x) + \xi)-\overline f(x, \alpha_i(x), \nabla\alpha_i(x))|,
$$
and
$$
q_j(x,  s, \xi) =
\left\{
\begin{array}{ll}
\overline f(x,  \beta_j(x), \xi) - \omega_{2,j}(x, s - \beta_j(x)),
&\mbox{ if } s > \beta_j(x),
\\
\overline f(x,  s, \xi),
&\mbox{ if } s \leq \beta_j(x),
\end{array}
\right.
$$
where
$$
\omega_{2,j}(x,  \delta) = \max_{|\xi| \leq \delta}
|\overline f(x,  \beta_j(x), \nabla\beta_j(x) + \xi)-\overline f(x, \beta_j(x), \nabla\beta_j(x))|,
$$
for $i \in \left\{1,...,k\right\}$ and $j \in \left\{1,...,l\right\}$.
At last, we define for a.e. $x \in \Omega$
and every $(s, \xi) \in \mathbb R \times \mathbb R^N$,
\[
F(x,  s, \xi) =
\left\{
\begin{array}{llrr}
\displaystyle
\max_{1 \leq i \leq k} p_i(x,  s, \xi),
&\mbox{ if }
s \leq \alpha(x),
\\[4mm]
\overline f(x,  s, \xi), &
\mbox{ if }
\alpha(x) < s < \beta(x),
\\[2mm]
\displaystyle
\min_{1 \leq j \leq l} q_j(x, s, \xi),
&\mbox{ if }
s \geq \beta(x).
\end{array}
\right.
\]
Then we consider the modified problem
\begin{equation}
\label{3-5.3}
\begin{array}{cl}
-\Delta u
= F(x,  u, \nabla u),
&\mbox{ in } \Omega,
\\
u = 0,
&\mbox{ on } \partial\Omega.
\end{array}
\end{equation}

Notice that  $F$ is a $L^{p}$-Carath\'eodory function and that
there exists $\gamma \in L^{p}(\Omega)$,
such that
$$
|F(x,  s, \xi)| \leq \gamma(x),
$$
for a.e. $x \in \Omega$ and every
$(s, \xi) \in \mathbb R \times \mathbb R^N$.
\medbreak

\noindent {\it Step 2. Every solution $u$ of \eqref{3-5.3}  satisfies $\alpha \leq u  \leq \beta$.} 
Let $u$    be a solution of \eqref{3-5.3}.  Assume by contradiction that $\displaystyle \min_{\overline\Omega}(u-\alpha)<0$. 
Let $i\in \{1, \ldots, k\}$ and  $\overline x \in \overline\Omega$ such that
$$
\min_{\overline\Omega}(u-\alpha)=\min_{\overline\Omega}(u-\alpha_i)=(u-\alpha_i)(\overline x) <0.
$$
Define $v=u-\alpha_i$.
As $v\geq 0$ on $\partial\Omega$ we have
 $\overline x \in  \Omega$. Therefore  $\nabla v(\overline x)=0$ and there is an open ball $B \subseteq \Omega$,
with $\overline x\in  B$  such that,
a.e. in $B$,
$$ 
|\nabla v(x)|\leq|v(x)|, \qquad
v(x)<0
$$ 
and
$$
\begin{array}{rcl}
-\Delta v &\geq& F(x,u(x), \nabla u(x))
- f(x,\alpha_i(x),\nabla\alpha_i(x))
\\
&\geq& \overline f(x,\alpha_i(x), \nabla u(x))
+\omega_{1i}(x,\alpha_i(x)-u(x))
- \overline f(x,\alpha_i(x),\nabla\alpha_i(x))
\\
&\geq& -\omega_{1i}(x,|\nabla v(x)|)
+\omega_{1i}(x,|v(x)|)  
\\
&\geq& 0,
\end{array}
$$
as $\omega_{1i}(x,\cdot)$ is increasing and $|v(x)|\geq |\nabla v(x)|$. This contradicts the strong maximum principle.
\medbreak

Similarly, one proves that $u \leq \beta$. 
\medbreak

\noindent {\it Step 3. Every solution of  \eqref{3-5.3} is
a solution of \eqref{De-He} and satisfies  $\alpha\leq u \leq
\beta$.} 
%The proof of this Step follows as in Theorem \ref{Thm 2-4.1soft}.
In Step 2, we proved that
every solution $u$ of \eqref{3-5.3} satisfies $\alpha\leq u \leq
\beta$ and hence is a solution of
$$
\begin{array}{cl}
-\Delta u = \overline f(x,u, \nabla u), & \mbox{in } \Omega,
\\
u=0,& \mbox{on } \partial\Omega.
\end{array}
$$
As $\overline f$ satisfies  \eqref{Hf}, we have
$\|u\|_{C^1(\overline\Omega)} < R$ and hence  $u$ is
a solution of \eqref{De-He}.
\smallbreak

\noindent {\it Step 4. Problem \eqref{3-5.3} has at least one
solution.} Let us consider the solution operator $\overline {\mathcal M}:
C^1(\overline\Omega) \to C^1(\overline\Omega)$
associated with \eqref{3-5.3},  which sends any function $u\in
C^1(\overline\Omega)$ onto the unique solution $v\in
W^{2,p}(\Omega)$ of
$$
\begin{array}{cl}
-\Delta v  = F(x,u, \nabla u), &  \mbox{in }\Omega,
\\
v=0,& \mbox{on } \partial\Omega.
\end{array}
$$
The operator $\overline {\mathcal M}$ is continuous, has a relatively compact range and its   fixed
points are the solutions of \eqref{3-5.3}. Hence there exists a constant  $ \overline R>0$, that we can suppose larger than $R$, such that,
for every $u \in C^1(\overline\Omega)$, 
$$
\|\overline {\mathcal M}u\|_{C^1(\overline\Omega)} <  \overline  R,
$$
and   hence (see, e.g., \cite{Ze86})
\begin{equation}
\label{3-5.4}
\deg (I-\overline {\mathcal M},B(0,\overline R)) = 1,
\end{equation}
where $I$ 
is the identity operator in $C^1(\overline\Omega)$ and $B(0,\overline R)$ is the open ball of center
$0$ and radius $\overline R$ in $C^1(\overline\Omega)$.  Therefore   
$\overline {\mathcal M}$ has a fixed point and problem \eqref{3-5.3} has at least one
solution. 
\medbreak

\noindent {\it Step 5. Problem \eqref{De-He} has at least one
solution.} By Step 4, we get the existence of a solution $u$ of the problem \eqref{3-5.3} 
and Step 2 implies that $u$ is a solution  of \eqref{De-He}
satisfying $\alpha\leq u \leq \beta$.
\medbreak

\noindent{\it Part 2. Existence of extremal solutions.} We know, from  Part 1,  
that the solutions $u$ of \eqref{De-He}, with $\alpha\leq u \leq
\beta$, are precisely  the fixed points of the solution operator   $\overline {\mathcal M}$    associated with  \eqref{3-5.3}. Set
$$
{\mathcal H} =\{u\in C^1(\overline\Omega)\mid u=\overline {\mathcal M}u\}.
$$
${\mathcal H}$   is a non-empty compact subset of   $C^1(\overline\Omega)$. Next,  for each $u\in {\mathcal H}$,  
define the closed set  ${\mathcal C}_u=\{z\in{\mathcal H} \mid z\leq u\}$. 
The  family $\{{\mathcal C}_u \mid u\in{\mathcal H}\}$ has the finite intersection property, as it follows from    Part 1 observing that  
if $u_1,  u_2\in{\mathcal H}$, 
then 
$\min\{u_1,u_2\}$ is an upper solution of \eqref{3-5.3} with $\alpha \le \min\{u_1,u_2\} $. 
Hence ${\mathcal C}_{u_1}\cap {\mathcal C}_{u_2}\not=\emptyset$.
By  the compactness of ${\mathcal H}$ there exists 
$v\in\bigcap_{u\in{\mathcal H}}{\mathcal C}_u$; clearly, $v$ is 
  the minimum solution in $[\alpha, \beta]$ of \eqref{De-He} in $\overline\Omega$.
\medbreak

\noindent {\it Part 3. Degree computation.} Now, let us assume that
$\alpha$ and $\beta$ are strict lower and   upper solutions
respectively. Since there exists a solution $u$ of \eqref{De-He}, 
which satisfies $\alpha\leq u \leq \beta$,  and every such  solution satisfies $\alpha
\ll u \ll \beta$, it follows that $\alpha
\ll  \beta$. Hence ${\mathcal S}$ is a non-empty open
set in $C^1(\overline\Omega)$ and there is no fixed point  either of ${\mathcal M}$ or of $\overline {\mathcal M}$  on its boundary  $\partial{\mathcal S}$.
Moreover, by  \eqref{3-5.2}, the sets of fixed points of ${\mathcal M}$ and $\overline {\mathcal M}$
coincide on ${\mathcal S}\cap B(0,R)$ and we have
$$
\deg(I-{\mathcal M}, {\mathcal S}\cap B(0, R)) =\deg(I-\overline {\mathcal M}, {\mathcal S}\cap B(0, R)).
$$
Furthermore, by the excision property of the 
degree (see, e.g., \cite{Ze86}),  we get from \eqref{3-5.2} and \eqref{3-5.4}  
$$
\deg(I-\overline {\mathcal M}, B(0,  R))=1.
$$
Finally, since all fixed points of $\overline {\mathcal M}$ are
in ${\mathcal S}\cap B(0, R)$, we conclude
$$
%\begin{array}[b]{rcl}
\deg(I-{\mathcal M}, {\mathcal S}\cap B(0, R))
%&=& 
=\deg(I-\overline {\mathcal M}, {\mathcal S}\cap B(0,
R))
%\\
%&=&
=\deg(I-\overline {\mathcal M}, B(0,  R)) =1.
%\end{array}
$$ 
This ends the proof.

\end{document}